\documentclass{amsart}

\usepackage[english]{babel}
\usepackage{amsmath}
\usepackage{amsfonts}
\usepackage{xypic}
\usepackage{amsthm}
\usepackage{mathrsfs}
\usepackage[all]{xy}
\usepackage{amssymb}
\usepackage{color}
\usepackage{latexsym}
\usepackage{amssymb}
\usepackage{graphicx}
\usepackage{lscape}

\numberwithin{equation}{section}

\newtheorem{theorem}{Theorem}[section]
\newtheorem{lemma}[theorem]{Lemma}
\newtheorem{proposition}[theorem]{Proposition}
\newtheorem{corollary}[theorem]{Corollary}

\theoremstyle{remark}
\newtheorem{remark}[theorem]{Remark}

\newcommand{\lie}{\mathfrak}
\newcommand{\nr}{\textnormal}
\newcommand{\wt}{\widetilde}

\newcommand{\ol}{\overline}

\newcommand{\CC}{\mathbb{C}}

\newcommand{\PP}{\mathbb{P}}

\newcommand{\RR}{\mathbb{R}}

\newcommand{\ZZ}{\mathbb{Z}}

\newcommand{\Cc}{\mathcal{C}}

\newcommand{\Ee}{\mathcal{E}}
\newcommand{\Ff}{\mathcal{F}}

\newcommand{\Hh}{\mathcal{H}}

\newcommand{\Mm}{\mathcal{M}}

\newcommand{\Oo}{\mathcal{O}}
\newcommand{\Pp}{\mathcal{P}}

\newcommand{\Ss}{\mathcal{S}}

\newcommand{\Uu}{\mathcal{U}}

\newcommand{\s}{\mathbf{s}}

\newcommand{\Fff}{\mathscr{F}}
\newcommand{\Mmm}{\mathscr{M}}

\newcommand{\aA}{\lie{a}}
\newcommand{\cC}{\lie{c}}

\newcommand{\gG}{\lie{g}}
\newcommand{\hH}{\lie{h}}

\newcommand{\mM}{\lie{m}}

\newcommand{\lL}{\lie{l}}
\newcommand{\pP}{\lie{p}}
\newcommand{\qQ}{\lie{q}}
\newcommand{\sS}{\lie{s}}
\newcommand{\tT}{\lie{t}}

\newcommand{\zZ}{\lie{z}}

\newcommand{\quotient}[2]{{\raisebox{.2em}{\thinspace $#1$}\left / \raisebox{-.2em}{ $#2$}\right.}}
\newcommand{\git}[2]{{\raisebox{.2em}{\thinspace $#1$}\left /\!\!/ \raisebox{-.2em}{$#2$}\right.}}
\newcommand{\morph}[6]{\begin{array}{cccc} #1 \, : &  #2  & \stackrel{#6}{\longrightarrow} &  #3  \\  & #4 &\longmapsto & #5  \end{array}}
\newcommand{\map}[5]{\begin{array}{cccc}   #1  & \stackrel{#5}{\lra} &  #2  \\  #3 & \longmapsto & #4  \end{array}}

\newcommand\Quotient[2]{
        \mathchoice
            {
                \text{\raise1ex\hbox{\thinspace $#1$}\Big{/} \lower1ex\hbox{$#2$} \thinspace}%
            }
            {
                #1\,/\,#2
            }
            {
                #1\,/\,#2
            }
            {
                #1\,/\,#2
            }
    }

\newcommand\GIT[2]{
        \mathchoice
            {
                \text{\raise1ex\hbox{\thinspace $#1$}\Big{/}\!\!\!\!\Big{/} \lower1ex\hbox{$#2$} \thinspace}%
            }
            {
                #1\,/\,#2
            }
            {
                #1\,/\,#2
            }
            {
                #1\,/\,#2
     a       }
    }

\newcommand{\dolbeault}{\overline{\partial}}
\newcommand{\Slash}{/\!\!/}

\newcommand{\lra}{\longrightarrow}

\DeclareMathOperator{\id}{id}
\DeclareMathOperator{\rk}{rk}
\DeclareMathOperator{\Ad}{Ad}
\DeclareMathOperator{\ad}{ad}
\DeclareMathOperator{\Lie}{Lie}
\DeclareMathOperator{\Sym}{Sym}

\DeclareMathOperator{\Pic}{Pic}

\DeclareMathOperator{\im}{im}

\DeclareMathOperator{\Spec}{Spec}

\DeclareMathOperator{\cay}{cay}

\DeclareMathOperator{\HN}{HN}

\DeclareMathOperator{\ord}{ord}
\DeclareMathOperator{\red}{red}

\newcommand{\dif}{\thinspace d}

\newcommand{\Sets}{(\mathbf{Sets})}
\newcommand{\Aff}{(\mathbf{Aff})}

\DeclareMathOperator{\GL}{GL}

\DeclareMathOperator{\SL}{SL}

\DeclareMathOperator{\U}{U}
\DeclareMathOperator{\SU}{SU}

\DeclareMathOperator{\Sp}{Sp}

\DeclareMathOperator{\Hom}{Hom}

\DeclareMathOperator{\Aut}{Aut}

\DeclareMathOperator{\C}{C}

\title{{\bf Higgs bundles over elliptic curves for real groups}}

\author{Emilio Franco}
\address{Emilio Franco \\ CMUP (Centro de Matem\'atica da Universidade do
Porto) \\ Universidade do Porto \\ Rua do Campo Alegre 1021/1055 \\
4169-007, Porto (Portugal)}
\email{emilio.franco@fc.up.pt}

\author{Oscar Garc\'{\i}a-Prada}
\address{Oscar Garc\'{\i}a-Prada \\ ICMAT (Instituto de Ciencias Matem\'aticas) \\ CSIC-UAM-UC3M-UCM \\ 
Calle Nicol\'as Cabrera 15 \\ 28049 Madrid (Spain)}
\email{oscar.garcia-prada@icmat.es}

\author{P. E. Newstead}
\address{P. E. Newstead \\ Department of Mathematical Sciences \\ University of Liverpool \\ 
Peach Street \\ Liverpool L69 7ZL (United Kingdom)}
\email{newstead@liverpool.ac.uk}

\begin{document}

\date{\today}



\thanks{First author supported by Funda\c{c}\~{a}o de Amparo \`a Pesquisa do Estado de S\~{a}o Paulo through FAPESP 2012/16356-6 and BEPE-2015/06696-2. Second author partially supported by the Ministerio de Econom\'ia y Competitividad of Spain through Project MTM2010-17717 and Severo Ochoa Excellence Grant.}

\begin{abstract}
We study topologically trivial $G$-Higgs bundles over an elliptic curve $X$ when the structure group $G$ is a connected real form of a complex semisimple Lie group $G^\CC$. We achieve a description of their (reduced) moduli space, the associated Hitchin fibration and the finite morphism to the moduli space of $G^\CC$-Higgs bundles.
\end{abstract}

\maketitle

\tableofcontents

\section{Introduction}

This is the third of a series of papers dedicated to the description of the moduli spaces of $G$-Higgs bundles over an elliptic curve $(X,x_0)$ (usually denoted simply by $X$). In the first paper \cite{franco&oscar&newstead_1} we dealt with the classical complex Lie groups and in the second \cite{franco&garcia-prada&newstead_2} with arbitrary connected complex reductive Lie groups, where we extended the description of the normalization of the moduli space given by Thaddeus \cite{thaddeus} to arbitrary degree. In this paper we address the case of real semisimple Lie groups.

To state our main result we first give some background. Let $G$ be a real semisimple Lie group, and let $H\subset G$ be a maximal compact subgroup. One has the  Cartan decomposition $\gG = \hH \oplus \mM$ of the Lie algebra of $G$,  where  $\hH$ is the Lie algebra of $H$ and $\mM$ is the orthogonal subspace to $\hH$ with respect to the Killing form. The group  $H$ acts on $\mM$ via the {\it isotropy representation} and this action extends to the complexification.  A {\it $G$-Higgs bundle} over the Riemann surface $\Sigma$ is a pair $(E,\Phi)$, where $E$ is a principal $H^\CC$-bundle over $\Sigma$ and $\Phi$ (the {\it Higgs field}) is a holomorphic section of $E(\mM^\CC) \otimes \Omega^1_{\Sigma}$ --- the bundle associated to the isotropy representation twisted by the canonical bundle of the curve. We say that $(E,\Phi)$ is topologically trivial if the topological class of the principal bundle $E$ is trivial. We denote the moduli space of topologically trivial $G$-Higgs bundles over $\Sigma$ by $\Mm_\Sigma(G)$, and by $\Mm^{\red}_\Sigma(G)$, its reduced subscheme.

\

The main result of this paper is the following (see Section \ref{section-moduli} for details).

\begin{theorem}[{\bf Theorem \ref{tm Mm cong olXi over W}}]\label{t1.1} 
Suppose that $X$ is an elliptic curve. Let $G$ be a connected real form of a complex semisimple Lie group, let $\Mm_X(G)$ be the moduli space of topologically trivial $G$-Higgs bundles over $X$ and let $\Mm^{\red}_X(G)$ be its reduced subscheme. Then there exists an isomorphism
\begin{equation} \label{eq G-HB} 
\Mm^{\red}_X(G) \cong \quotient{\Xi_G}{W},
\end{equation}
where $W$ is the Weyl group for the action of $H^\CC$ on  $\mM^\CC$, and $\Xi_G$ is a quasiprojective variety described as a fibration of (finite quotients of) abelian Lie algebras over $X \otimes_\ZZ \Lambda$, where $\Lambda$ is the cocharacter lattice of $H^\CC$. 
\end{theorem}

The (reduced) moduli space $\Mm^{\red}_X(G)$ is not irreducible in general and we describe its irreducible components (see Remark \ref{rm irreducible components of Mm}).

\

Higgs bundles were introduced by Hitchin in the context of vector bundles in \cite{hitchin-selfuality_equations} and generalised to arbitrary complex reductive Lie groups in  \cite{hitchin-duke}. Simpson \cite{simpson1, simpson2} constructed the moduli space $\Mm_\Sigma(G^\CC)$ of topologically  trivial $G^\CC$-Higgs bundles for any complex reductive Lie group.  A major result of the theory of Higgs bundles is the non-abelian Hodge theory correspondence, proved by Hitchin \cite{hitchin-selfuality_equations}, Do\-naldson \cite{donaldson}, Simpson \cite{simpson-hb&ls, simpson1, simpson2} and Corlette \cite{corlette}. This states the existence of a homeomorphism, $\Mm_\Sigma(G^\CC) \cong \Cc_\Sigma(G^\CC)$, between  the moduli spaces of $G^\CC$-Higgs bundles and  flat $G^\CC$-connections. 

Higgs bundles for real groups were already considered by Hitchin in \cite{hitchin-selfuality_equations,hitchin-topology} and Simpson \cite{simpson-hb&ls}. An intrinsic approach and systematic study has been done by the second author in collaboration with  Bradlow, Gothen and Mundet i Riera (see e.g. \cite{bradlow&garcia-prada&gothen,BGG2,BGG3,oscar&ignasi&gothen,GGM2}). The existence of the moduli space $\Mm_\Sigma(G)$ for any real semisimple Lie group $G$ follows from Schmitt \cite{schmitt}. 

If $G$ is a real form of a complex semisimple Lie group $G^\CC$,  a $G$-Higgs
bundle extends naturally to a $G^\CC$-Higgs bundle, defining a finite morphism
(see \cite[Prop. 5.8]{garcia-prada&ramanan}). 
\begin{equation} \label{eq MmG to MmG^C}
\Mm_\Sigma(G) \rightarrow \Mm_\Sigma(G^\CC).
\end{equation}


For classical complex Lie groups, using spectral curves, Hitchin \cite{hitchin-duke} showed that $\Mm_\Sigma(G^\CC)$ fibres over a vector space $B_{G^\CC}$ with abelian varieties as generic fibres, becoming an algebraically completely integrable system. This is the so-called {\it Hitchin fibration}. A more canonical  definition of the Hitchin fibration was provided by Donagi \cite{donagi} giving an intrinsic definition of  the Hitchin base $B_{G^\CC}$ as the space $H^0(\Sigma, \gG^\CC \otimes \Omega^1_\Sigma \Slash G^\CC)$ of cameral covers of $\Sigma$. One can generalize the Hitchin fibration to the case of real groups (see \cite{peon,schaposnik,hitchin&schaposnik,garcia-prada&peon&ramanan}) 
\begin{equation} \label{eq hitchin fibration}
\Mm_\Sigma(G) \longrightarrow B_{G},
\end{equation} 
where, using Donagi's approach, $B_{G}$ is $H^0(\Sigma, \mM^\CC \otimes
\Omega^1_\Sigma \Slash H^\CC)$. As shown in 
\cite{franco1,schaposnik,peon,hitchin&schaposnik}, the generic fibre of \eqref{eq hitchin fibration} is no longer an abelian variety for certain groups.

The canonical bundle of an elliptic curve $X$ is trivial, $\Omega^1_X \cong \Oo_X$, and therefore, in this case, the Higgs field is simply an element of $H^0(X,E(\mM^\CC))$. Working with elliptic curves allows a greater level of explicitness. Atiyah \cite{atiyah} described the moduli space of vector bundles, 
\[
M_X(\GL(n,\CC)) \cong \Sym^n(X), 
\]
while Laszlo \cite{laszlo} and Friedman, Morgan and Witten \cite{friedman&morgan, friedman&morgan&witten}, gave a description of the moduli space $M_X(G^\CC)$ of principal $G^\CC$-bundles for a complex reductive Lie group: 
\[
M_X(G^\CC) \cong \quotient{(X \otimes_\ZZ \Lambda)}{W},
\]
where $\Lambda$ is the cocharacter lattice and $W$ is the Weyl group of $G^\CC$. In \cite{thaddeus} Thaddeus studied the case of $G^\CC$-Higgs bundles, showing that the normalization of the moduli space is
\begin{equation} \label{eq G^C-HB}
\widetilde{\Mm}_X(G^\CC) \cong \quotient{(T^*X \otimes_\ZZ \Lambda)}{W}.
\end{equation}
In \cite{franco&garcia-prada&newstead_2} the authors studied the case of $G^\CC$-Higgs bundles, extending this description to arbitrary degree. Since in the case of real groups, the moduli spaces of $G$-Higgs bundles are no longer bijective with their normalizations, we pursue the description of the reduced subscheme $\Mm^{\red}_X(G) \subset \Mm_X(G)$, which we call the reduced moduli space.

The results of this paper are structured as follows. In Section \ref{sc preliminaries} we provide the preliminaries needed for our work, a review on Lie theory, $G$-Higgs bundles and principal bundles over elliptic curves. We show in Section \ref{sc stability} that, if a $G$-Higgs bundle is semistable (resp. polystable), then the underlying principal $H^\CC$-bundle is semistable (resp. polystable), and this allows us to describe these objects explicitly. In Section \ref{sc FfG} we generalize to real groups the existence of what Simpson \cite{simpson1, simpson2} calls {\em representation space} for Higgs bundles and relate it to the  moduli space $\Mm_X(G)$. We also prove some properties of the representation spaces specific to the elliptic case which, together with the explicit description of polystable $G$-Higgs bundles given in Section \ref{sc stability}, allow us to prove the main result of the article, Theorem \ref{t1.1} (Theorem \ref{tm Mm cong olXi over W}). If $G$ is complex we  recover the description  \eqref{eq G^C-HB} given in \cite{thaddeus}. In Section \ref{sc involution}, we study in detail the involution $\imath_G$ (Proposition \ref{pr description of i_G}) and the morphism \eqref{eq MmG to MmG^C} (Proposition \ref{pr description of embedding}). In Section \ref{sc the hitchin fibration} we provide a complete description of the Hitchin fibration \eqref{eq hitchin fibration} and its generic fibre, which is isomorphic to a finite quotient of a certain subset of the abelian variety $X \otimes_\ZZ \Lambda$ (Corollary \ref{co generic Hitchin fibre}). In particular, when $G = \SU^*(4)$ we obtain that the generic Hitchin fibre is $\PP^1 \times \PP^1$, illustrating our comment above regarding the fact that the generic fibre is not always an abelian variety. Finally, we study the Hitchin equation in Section \ref{sc hitchin equation} showing that it decouples into one equation for the metric and another equation for the Higgs field (Proposition \ref{pr splitting of the hitchin equation}). We observe, then, that the Hitchin--Kobayashi correspondence follows in the elliptic case from the Narasimhan--Seshadri--Ramanathan Theorem \cite{narasimhan&seshadri,ramanathan_stable}, and the stability results of Section \ref{sc stability}.

We work in the category of algebraic schemes over $\CC$. Unless otherwise stated, all the bundles considered are algebraic bundles. By bijective morphism we understand an algebraic morphism between schemes that induces a bijection on the sets of $\CC$-points.


\section{Preliminaries}
\label{sc preliminaries}

\subsection{Lie groups and Lie algebras}

\subsubsection{The Cartan decomposition}
\label{sc cartan subalgebras}

Let $G$ be a connected semisimple real form of a connected complex semisimple Lie group $G^\CC$. Let $H\subset G$ be a maximal compact subgroup and let $\gG = \hH \oplus \mM$ be the associated Cartan decomposition, where $\hH$ is the Lie algebra of $H$ and $\mM$ is the orthogonal subspace to $\hH$ with respect to the Killing form. This defines an involution $\theta: \gG \to \gG$ --- the Cartan involution ---, by $\theta |_{\hH} = \id_\hH$ and $\theta |_{\mM} = - \id_\mM$, which can be
naturally extended to the complexification $\gG^\CC$, giving the decomposition $\gG^\CC = \hH^\CC \oplus \mM^\CC$.

This satisfies
\[
[\lie{h},\lie{h}] \subset \lie{h}, \qquad [\lie{h},\lie{m}] \subset \lie{m}, \qquad [\lie{m},\lie{m}] \subset \lie{h}.
\]
Therefore, the restriction to $H$ of the adjoint representation reduces to an action on $\lie{m}$, which extends to what we will call the {\em isotropy representation} of $H^\CC$ on $\lie{m}^\CC$
\[
\iota : H^\CC \to \GL(\mM^\CC).
\]

Consider a Cartan subalgebra $\cC^\CC$ of $\gG^\CC$ and denote by $R(\gG^\CC,\cC^\CC)$ the set of roots associated to $\cC^\CC$. One has the root-space decomposition 
\[
\gG^\CC = \cC^\CC \oplus \bigoplus_{\alpha \in R(\gG^\CC,\cC^\CC)} (\gG^\CC)^\alpha.
\]

Recall that the choice of a lexicographic order on some basis of $\cC^\CC$ defines a notion of positivity for the roots. We write $R^+(\gG^\CC, \cC^\CC)$ for the set of positive roots with respect to a given lexicographic order. A positive root $\alpha \in R^+(\gG^\CC,\cC^\CC)$ is simple if it cannot be written as a sum of any two other positive roots. We denote the set of simple roots by $\Delta(\gG^\CC, \cC^\CC) \subset R^+(\gG^\CC, \cC^\CC)$.

A {\it Cartan subalgebra of the real Lie algebra} $\gG$ is a subalgebra whose complexification is a Cartan subalgebra of $\gG^\CC$. It is always possible to find a $\theta$-stable Cartan subalgebra $\cC$ of $\gG$. In that case, if $\alpha \in R(\gG^\CC, \cC^\CC)$ then $\theta\alpha := \alpha \circ \theta$ is also in $R(\gG^\CC, \cC^\CC)$. Also we have that $\cC = (\cC \cap \hH) \oplus (\cC \cap \mM)$. The number $\dim_\RR(\cC \cap \hH)$ is the compact dimension of $\cC$, while $\dim_\RR(\cC \cap \mM)$ is the non-compact dimension of $\cC$. A Cartan subalgebra is said to be {\it maximally compact} or {\it maximally non-compact} if it maximizes the compact or the non-compact dimension among $\theta$-stable Cartan subalgebras.

\begin{remark}
One can always construct a maximally compact Cartan subalgebra. Take the Lie algebra $\tT$ of a ma\-xi\-mal torus $T$ of $H$. Taking $\aA_0$ to be a maximal abelian subspace of $\zZ_{\mM}(\tT)$, one has that $\cC_0 = \tT \oplus \aA_0$ is a maximally compact $\theta$-stable Cartan subalgebra.
\end{remark}

\begin{remark}
\label{rm Delta preserved by theta}
Given a maximally compact $\theta$-stable Cartan subalgebra $\cC_0$ of $\gG$, one can always find a lexicographic order such that $\theta$ preserves the set of positive roots, and therefore the set of simple roots. It suffices to define a lexicographic order in terms of a basis of $i \tT = i(\cC_0 \cap \hH)$ followed by a basis of $\aA_0 = (\cC_0 \cap \mM)$. By \cite[Proposition 6.70]{knapp} there are no roots that vanish entirely on $\tT = (\cC_0 \cap \hH)$ when the Cartan subalgebra is maximally compact. Since $\theta$ is $1$ on $\cC_0 \cap \hH$ and $-1$ on $\cC_0 \cap \mM$, we see that the positivity will be preserved under $\theta$.
\end{remark}

We write $R(\gG, \cC) = R(\gG^\CC, \cC^\CC)$ for the set of roots associated to the $\theta$-stable Cartan subalgebra $\cC^\CC$ restricted to $\cC$. One can check that the evaluation of a root $\alpha \in R(\gG, \cC)$ is imaginary on $(\cC \cap \hH)$, and real on $(\cC \cap \mM)$. Consequently, a root is {\it real} if it vanishes on $(\cC \cap \hH)$, {\it imaginary} if it vanishes on $(\cC \cap \mM)$, and {\it complex} otherwise. Recall that, for every root $\alpha \in R(\gG, \cC)$, one has that $\theta \alpha \in R(\gG,\cC)$ is a root too. If $\alpha$ is imaginary, $\theta \alpha = \alpha$, so $\theta(\gG^\CC)^\alpha = (\gG^\CC)^{\theta \alpha} = (\gG^\CC)^\alpha$ and the root-space $(\gG^\CC)^\alpha$ is $\theta$-stable. Since the root-spaces are $1$-dimensional, either $(\gG^\CC)^\alpha \subset \hH^\CC$ and $\alpha$ is said to be {\it compact}, or $(\gG^\CC)^\alpha \subset \mM^\CC$ and $\alpha$ is said to be {\it non-compact}. We write $I_{cp}(\gG,\cC), I_{nc}(\gG,\cC) \subset R(\gG, \cC)$ for the sets of imaginary compact and non-compact roots, $R_{re}(\gG,\cC) \subset R(\gG, \cC)$ for the subset of real roots and $R_{cx}(\gG,\cC)$ for the subset of pairs $\{ \alpha, \theta\alpha \}$ of complex roots. If $\cC_0$ is maximally compact, one can prove that $R_{re}(\gG,\cC_0) = 0$ (see Section \ref{sc cayley transform}), and therefore one has the decomposition
\begin{equation} \label{eq decomposition of R}
R(\gG, \cC_0) = I_{cp}(\gG,\cC_0) \sqcup I_{nc}(\gG,\cC_0) \sqcup R_{cx}(\gG, \cC_0).
\end{equation}

Let $\cC_0 = \tT \oplus \aA_0$ be a maximally compact Cartan subalgebra. Let $C_0^\CC$ be the Cartan subgroup of $G^\CC$ associated to $\cC_0^\CC$ and let $T^\CC$ be the Cartan subgroup of $H^\CC$ associated to $\tT^\CC$. One should consider two Weyl groups,
\begin{equation} \label{eq definition of Y}
Y := W(G^\CC,C_0^\CC) = W(\gG^\CC,\cC_0^\CC)
\end{equation}
and 
\begin{equation} \label{eq definition of W}
W := W(H^\CC,T^\CC) = W(\hH^\CC,\tT^\CC).
\end{equation}

Since $\aA_0^\CC = \zZ_{\mM^\CC}(\tT^\CC)$, the normalizer $N_{H^\CC}(\tT^\CC)$ also normalizes $\aA_0^\CC$ and therefore $\cC_0^\CC$. This implies that
\[
N_{H^\CC}(T^\CC) = N_{G^\CC}(C_0^\CC) \cap H^\CC
\]
where $C_0^\CC$ is the Cartan subgroup of $G^\CC$ with Lie algebra $\cC_0^\CC$. Then there is a well defined map of Weyl groups,
\[
W \lra Y,
\]
which is injective since its kernel is the projection of $N_{H^\CC}(\cC_0^\CC) \cap \exp(\aA_0^\CC)$ and therefore it is trivial. 

\begin{remark} \label{rm W preserves I_nc}
Considered as a subgroup of $Y$, $W$ is the group that preserves the splitting $\cC_0^\CC = \tT^\CC \oplus \aA_0^\CC$, and therefore, $W$ can be described as the subgroup of $Y$ that preserves the decomposition \eqref{eq decomposition of R}.
\end{remark}

\subsubsection{Strongly orthogonal roots and real Cartan subalgebras}
\label{sc cayley transform}

Take an ima\-gi\-na\-ry non-compact root $\alpha \in I_{nc}(\gG,\cC)$ and let $\{ x_\alpha : x_\alpha \in (\gG^\CC)^\alpha \}_{\alpha \in R(\gG, \cC_0)}$ be a set of (non-zero) representatives of the root-spaces closed under the Lie bracket (i.e. for every two $x_{\alpha_1}, x_{\alpha_2}$ contained in our set, one has that $[x_{\alpha_1}, x_{\alpha_2}]$ is contained in the set too). Since $\alpha$ is imaginary, the complex conjugate $\ol{x}_\alpha$ is contained in $(\gG^\CC)^{-\alpha}$. For $x_\alpha \in (\gG^\CC)^\alpha$, one can define the {\it first Cayley transform}  associated to $\alpha$ as 
\[
\cay_{1,\alpha} := \Ad(\exp \frac{\pi}{4}(\ol{x}_\alpha-x_\alpha)) : \gG^\CC \to \gG^\CC.
\]
Given a Cartan subalgebra $\cC$ of $\gG$ with compact dimension $n$, the Cayley transform gives us a new Cartan subalgebra $\cC'$ of compact dimension $n-1$
\[
\cC' := \gG \cap \cay_{1,\alpha}(\cC^\CC) = \ker(\alpha |_{\cC}) \oplus \RR(x_\alpha + \ol{x}_\alpha).
\]

\begin{lemma}
\label{lm different x_alpha give different Cartan subalgebras}
Let $\alpha \in I_{nc}(\gG,\cC_0)$ and take $x_\alpha \in (\gG^\CC)^\alpha$. Each choice of $\pm e^{i \theta} \in \U(1)$ gives a different Cartan subalgebra $\ker(\alpha |_{\cC_0}) \oplus \RR (e^{i \theta}x_\alpha + \ol{e^{i \theta}x}_\alpha)$. All these Cartan subalgebras are conjugate under the action of $T$.
\end{lemma}

\begin{proof}
Note that, unless $e^{i \theta} = \pm 1$, $(x_\alpha + \ol{x}_\alpha)$ and $(e^{i \theta}x_\alpha + \ol{e^{i \theta}x}_\alpha)$ do not commute:
\begin{align*}
[x_\alpha + \ol{x}_\alpha, e^{i \theta}x_\alpha + \ol{e^{i \theta}x}_\alpha] =& [x_\alpha + \ol{x}_\alpha, e^{i \theta}x_\alpha + e^{-i \theta}\ol{x}_\alpha]
\\
=& [\ol{x}_\alpha, e^{i \theta} x_\alpha] + [x_\alpha, e^{-i \theta}\ol{x}_\alpha]
\\
=& e^{i \theta} (1 - e^{-2 i \theta})[\ol{x}_\alpha, x_\alpha]
\\
\neq & 0,
\end{align*}
so each gives a different Cartan subalgebra. 

The elements of $T$ have the form $g = \exp s$ with $s \in \tT$. Since $\alpha(s) = i \theta$ is an imaginary number, one has that $\ad_g(x_\alpha) = e^{i \theta} x_\alpha$, so all these Cartan subalgebras are conjugate by some element of $T$.
\end{proof} 

\begin{remark} \label{rm ZZ_2 are the automorphisms of cC_alpha given by T}
From the proof of Lemma \ref{lm different x_alpha give different Cartan subalgebras}, we observe that the only automorphisms of $\cC' = \gG \cap \cay_{1,\alpha}(\cC^\CC)$ given by elements of $T$ are $\{ 1, -1 \}$ acting on $\RR(x_\alpha + \ol{x}_\alpha)$.
\end{remark}

On the other hand, given a real root $\alpha \in R_{re}(\gG,\cC)$ of $\cC$ and a non-zero element $x_\alpha \in (\gG^\CC)^\alpha$ of its associated root-space one has that $\theta x_\alpha$ is contained in $(\gG^\CC)^{-\alpha}$. Taking $x_\alpha$, one can define the {\it second Cayley transform}
\[
\cay_{2,\alpha} := \Ad(\exp \frac{\pi i}{4}(\theta x_\alpha-x_\alpha)) : \gG^\CC \to \gG^\CC.
\]
Given a Cartan subalgebra $\cC$ of $\gG$ with compact dimension $n$, the Cayley transform give us a new Cartan subalgebra $\cC''$ of compact dimension $n+1$
\[
\cC'' := \gG \cap \cay_{2,\alpha}(\cC^\CC) = \ker(\alpha |_{\cC}) \oplus \RR(x_\alpha + \theta x_\alpha).
\]
See \cite[Section VI.7]{knapp} for a detailed description of the Cayley transform.

We say that $\alpha,\beta \in R(\gG^\CC,\hH^\CC)$ are {\it strongly orthogonal} if 
\[
\alpha + \beta \notin R(\gG^\CC,\hH^\CC) \quad \text{and} \quad \alpha - \beta \notin R(\gG^\CC,\hH^\CC). 
\]
The importance of this definition lies on the fact that one can repeatedly apply the Cayley transform $\cay_1$ if we have a set of mutually strongly orthogonal roots. In order to make this statement clearer, set $I^+_{nc}(\gG,\cC_0) := I_{nc}(\gG,\cC_0) \cap R^+(\gG,\cC_0)$ to be the set of imaginary non-compact roots that are positive with respect to a certain lexicographic order. Following \cite{sugiura}, we say that the subset $B \subset I^+_{nc}(\gG,\cC_0)$ is an {\it admissible root system} if any two roots $\beta_i, \beta_j$ in $B$ are strongly orthogonal. 

Fix a maximally compact $\theta$-stable Cartan subalgebra $\cC_0$ and a non-zero element $x_\beta \in (\gG^\CC)^\beta$ for each $\beta \in \Delta(\gG,\cC_0)$ such that, for each element of the Weyl group $\omega \in W(\gG^\CC, \cC_0^\CC)$, one has that $\omega \cdot x_\beta = x_{\omega \cdot \beta}$ for each $\beta \in \Delta(\gG,\cC_0)$. For the admissible root system $B = \{\beta_1, \dots, \beta_\ell \}$, we define 
\[
\cay_B := \cay_{1,\beta_\ell} \circ \dots \circ \cay_{1,\beta_2} \circ \cay_{1,\beta_1}. 
\]
Using $\cay_B$, we set
\begin{equation} \label{eq definition of cC_B}
\cC_B := \gG \cap \cay_B (\cC_0^\CC),
\end{equation}
which can be described as
\begin{equation}
\label{ec Cartan subalgebra obtained from Cayley 1}
\cC_B = \bigcap_{\beta \in B} \ker (\beta|_{\cC_0}) \oplus \bigoplus_{\beta \in B} \RR(x_\beta + \ol{x}_\beta).
\end{equation}
One can check that a permutation of the sequence of Cayley transforms gives the same Cartan subalgebra. We write
\begin{equation} \label{eq definition of tT_B}
\tT_B := \cC_B \cap \hH = \bigcap_{\beta \in B} \ker (\beta |_\tT)
\end{equation}
and
\begin{equation} \label{eq definition of aA_B}
\aA_B := \cC_B \cap \mM = \aA_0 \oplus \bigoplus_{\alpha \in B} \RR(x_\beta + \ol{x}_\beta).
\end{equation}

Associated to a fixed $\theta$-stable maximally compact Cartan subalgebra $\cC_0$, we denote by $\Upsilon$ the set of all admissible root systems, that is
\begin{equation} \label{eq definition of Upsilon}
\Upsilon = \{ B \subset I^+_{nc}(\gG,\cC_0) \nr{ where every } \beta_i,\beta_j \in B \nr{ are strongly orthogonal}\}.
\end{equation}
We include the zero set in our definition, $\{ 0 \} \in \Upsilon$.

Take the Weyl groups $Y$ and $W$ given in \eqref{eq definition of Y} and \eqref{eq definition of W}, and recall from Remark \ref{rm W preserves I_nc} that $W$ can be understood as the subgroup of $Y$ that preserves the set of non-compact roots $I_{nc}(\gG,\cC_0)$. We say that two admissible root systems $B_1$ and $B_2$ are {\it conjugate} if there exists an element $\omega \in W$ such that 
\[
B_2 = \omega \cdot B_1.
\]

Conjugacy classes of admissible root systems classify real Cartan subalgebras.

\begin{lemma}[\cite{sugiura} Corollary 2 to Theorem 3 and Theorem 6] 
\label{lm Cayley transform give all Cartan subalgebras}
Every $\theta$-stable Cartan subalgebra of $\gG$ is conjugate by $H$ to $\cC_B$ for some admissible root system $B \in \Upsilon$. Furthermore, the set of conjugacy classes of admissible root systems $B \in \Upsilon$ is in one to one correspondence with the set of conjugacy classes of Cartan subalgebras of $\gG$. 
\end{lemma}

\begin{remark}
We find in \cite{sugiura} a case by case description of admissible root systems for all real semisimple Lie algebras. 
\end{remark}

Denoting by $|A|$ the cardinality of a set $A$, we say that the admissible root system $D$ is {\it maximal} in $\Upsilon$ if $|D| \geq |B|$ for every $B \in \Upsilon$. Taking $D$ maximal, one has that $\cC = \tT \oplus \aA$ is a maximally non-compact $\theta$-stable Cartan subalgebra and $\aA$ is a maximal abelian subspace of $\mM$.


\subsubsection{Some results on Weyl groups}

Given an admissible root system $B \in \Upsilon$ and its associated Cayley transform $\cay_B$, we set $Y_B$ to be the Weyl group $W(\gG^\CC, \cC_B^\CC)$ associated to $\cC_B^\CC$. Note that
\begin{equation} \label{eq definition of Y_B}
Y_B = \cay_B \circ \, Y \circ \cay_B^{-1},
\end{equation}
and, obviously, $Y_B$ is isomorphic to $Y$. 

\begin{remark} \label{rm W_B is the normalizer of aA_B}
Recall that $W$ can be understood as the subgroup of $Y$ that preserves the splitting $\cC_0^\CC = \tT^\CC \oplus \aA_0^\CC$. Then, the subgroup of $Y_B$ that preserves the splitting $\cC_B^\CC = \tT_B^\CC \oplus \aA_B^\CC$ is contained in the image of $W$. In other words, 
\[
N_{Y_B}(\tT_B^\CC) = N_{Y_B}(\aA_B^\CC) \subset \cay_B \circ \, W \circ \cay_B^{-1}. 
\]
\end{remark}

For any admissible root system $B \in \Upsilon$, define the group
\begin{equation}
\label{eq definition of Gamma_B}
\Gamma_B := \prod_{\alpha \in B} \alpha \bullet \{ 1, -1 \},
\end{equation}
where $\bullet$ denotes a formal product. We can naturally define an action of $\Gamma_B$ on $\aA_B$ by setting the element $\alpha \bullet 1$ to act trivially on $\aA_B$, while $\alpha \bullet (-1)$ acts trivially on $\aA_0 \oplus \bigoplus_{\alpha \in B, \beta \neq \alpha} \RR(x_\beta + \ol{x}_\beta)$ and sends
\begin{equation}
\label{eq action of Gamma on aA_B}
\map{\RR(x_\alpha + \ol{x}_\alpha)}{\RR(x_\alpha + \ol{x}_\alpha)}{s}{-s.}{}
\end{equation}

\begin{remark}
\label{rm Gamma_B are the automorphisms of cC_B given by T}
It follows from Remark \ref{rm ZZ_2 are the automorphisms of cC_alpha given by T} that $\Gamma_B$ is the group of automorphisms of $\cC_B$ given by conjugation by $T$.
\end{remark} 

Take  
\begin{equation} \label{eq definition of Z_W B}
Z_W(B) = \{ \omega \in W \text{ such that } \omega \cdot \beta \in B \text{ for all } \beta \in B \},
\end{equation}
and note that one can define an action of $Z_W(B)$ on $\Gamma_B$ as follows,
\[
\omega \cdot (\alpha \bullet (\pm 1)) = (\omega \cdot \alpha) \bullet (\pm 1).
\]
Define $\Gamma_B \rtimes Z_W(B)$ and note that it acts naturally on $\aA_B$, where the action of $Z_W(B)$ on $\aA_B = \aA_0 \oplus \bigoplus_{\alpha \in B} \RR(x_\alpha + \ol{x}_\alpha)$ is given by the natural action of $W$ on $\aA_0$ and sending the ray $\RR(x_\alpha + \ol{x}_\alpha)$ to the ray $\RR(x_{\omega \cdot \alpha} + \ol{x}_{\omega \cdot \alpha})$. 
 
\begin{lemma}
\label{lm W_B cong Gamma_B rtimes Z_W B}
We can identify the semidirect product $\Gamma_B \rtimes Z_W(B)$ with a subgroup of $N_{Y_B}(\aA_B)$ and the action of $N_{Y_B}(\aA_B)$ on $\aA_B$ is completely determined by the action of this subgroup.
\end{lemma} 
 
\begin{proof}
Note that $\Gamma_B$ is a commutative normal subgroup of $Y_B$, since it is given by the reflections associated to strictly-orthogonal (and thereore orthogonal) roots $\beta_1, \dots, \beta_\ell \in B$ and also $\Gamma_B$ preserves $\aA_B$, so $\Gamma_B \subset N_{Y_B}(\aA_B)$. Then, it is clear that the subgroup generated by $\Gamma_B$ and $\cay_B \circ \, Z_W(B) \circ \cay_B^{-1}$ is a subgroup of $N_{Y_B}(\aA_B)$. We can identify this subgroup with $\Gamma_B \rtimes Z_W(B)$. 

By Remark \ref{rm W_B is the normalizer of aA_B}, every element of $N_{Y_B}(\aA_B)$ is of the form $\cay_B \circ \, \omega \circ \cay_B^{-1}$ with $\omega \in W$. The action of this element sends the rays of the form $\RR(x_\alpha + \ol{x}_\alpha)$ with $\alpha \in B$ to the rays $\RR(x_{\omega \cdot \alpha} + \ol{x}_{\omega \cdot \alpha})$, possibly changing the orientation of the ray. Then, $\omega$ is contained in $Z_W(B)$ since $\omega \cdot \alpha \in B$ as well. Then the subgroup generated by $\Gamma_B$ and $\cay_B \circ \, Z_W(B) \circ \cay_B^{-1}$ is the whole $N_{Y_B}(\aA_B)$, and the result follows.
\end{proof}

\subsubsection{Parabolic subgroups and antidominant characters}

Let $H^\CC$ be a complex reductive Lie group with Lie algebra $\hH^\CC$ being the complexification of some compact Lie algebra $\hH$ and take a Cartan subalgebra $\tT^\CC = \zZ_{\hH^\CC}(\hH^\CC) \oplus \tT_{ss}^\CC$ of $\hH^\CC$. Let $\langle \, , \, \rangle$ be the Killing form extended to $\hH^\CC$.

For any subset $A \subset \Delta(\hH^\CC,\tT^\CC)$ we define $R_A$ to be the subset of $R(\hH^\CC, \tT^\CC)$ whose elements have the form $\alpha = \Sigma_{\beta \in \Delta} m_\beta \beta$ with $m_\beta \geq 0$ for all $\beta \in A$. We define the subalgebra
\[
\lie{p}_A: = \tT^\CC \oplus \bigoplus_{\delta \in R_A} (\hH^\CC)^{\delta}.
\]
We call the connected subgroup $P_A \subset H^\CC$ with Lie algebra $\lie{p}_A$, a {\it standard parabolic subgroup} and we refer to any subgroup conjugate to $P_A$ as a {\it parabolic subgroup}. 

We define $R^0_A \subset R_A$ as the set of roots of the form $\alpha = \sum_{\beta \in \Delta} m_\beta \beta$ such that $m_\beta = 0$ for every $\beta \in A$. Take the subalgebra of $\lie{p}_A$,
\[
\lL_A: = \zZ^\CC \oplus \tT^\CC \oplus \bigoplus_{\alpha \in R^0_A} (\hH^\CC)^\alpha.
\]
The Lie connected subgroup $L_A \subset P_A$ with Lie algebra $\lL_A$ is called the {\it Levi subgroup} of $P_A$. 


Let $\Lambda_Z$ be the kernel of the exponential map restricted to $\zZ_{\hH^\CC}(\hH^\CC)$. Define $\zZ_\RR$ to be $\Lambda_Z \otimes_\ZZ \RR \subset \zZ_{\hH^\CC}(\hH^\CC)$ and consider the dual space $\Hom_\RR(\zZ_\RR,i\RR)$. 
For every $\alpha \in R(\hH^\CC, \tT^\CC)$ we define its {\it coroot} as $\hat{\alpha} := \frac{2 \alpha^*}{\langle \alpha , \alpha \rangle}$, where $\alpha^* \in \tT^\CC$ is such that $\alpha = \langle \cdot, \alpha^* \rangle$. Taking the dual with respect to the Killing form of $\hat{\alpha}$, we define $\lambda_\alpha$ to be the {\it fundamental weight} associated to $\alpha$. Note that the fundamental weights are elements of $(\tT^\CC)^*$ and, by construction, $\lambda|_{\zZ_{\hH^\CC}(\hH^\CC)} = 0$. 

Let $A$ be a subset of $\Delta(\hH^\CC, \tT^\CC)$ and let $\pP_A$ be the standard parabolic subalgebra associated to it. An {\it antidominant character} of $\pP_A$ is any element of $(\tT^\CC)^*$ of the form
\[
\chi = \delta + \sum_{\alpha \in A} n_\delta \lambda_\alpha,
\]
where $\delta \in \Hom(\zZ_\RR,i\RR)$ and each $n_\alpha$ is non-positive. If further we have $n_\alpha < 0$ for every $\alpha \in A$, the character is {\it strictly antidominant}. 

\begin{remark} \label{rm exponentiation of characters}
Since $L_A$ is a connected complex reductive Lie group, denoting by $Z_L$ the connected component of its centre, and $L_A^{ss} = [L_A, L_A]$ its semisimple part, one has $L_A \cong Z_L \times_F L_A^{ss}$, where $F$ is some finite group. This gives a map $L_A \to Z_L/F$, and composing with $P_A \to L_A$, one has the morphism of Lie groups \[
\pi_A : P_A \to Z_L / F.
\]
It the follows that not every character of the Lie algebra $\pP_A$ exponentiates to the associated parabolic group $P_A$, but, since $F$ is finite, we know that for every character $\chi$ of $\pP_A$, there exists $n \in \ZZ$ such that $\chi^n$ exponentiates to a character of the group $P_A$. Then, the characters of $\pP_A$ which exponentiate generate (as a subset of a vector space) the space of all characters of $\pP_A$.
\end{remark}

To any character $\chi$, we associate $s_\chi \in \tT^\CC$, its representative via the Killing form. Note that the roots of $\hH^\CC$ take pure imaginary values on $\hH$ since $\ad h$ with $h \in \tT$ is skew-symmetric with respect to the Killing form. This ensures that $s_\chi$ belongs to $i\hH_{ss}$.

\begin{lemma}
\label{lm pPA in pP- y lLa in lL-}
Let $s \in i \hH_{ss}$. Define the sets
\begin{align*}
& \pP_{s}: = \{ x \in \hH^\CC \nr{ such that } \Ad(e^{t s} ) x \nr{ remains bounded as } \RR \ni t \to \infty \},
\\
& \lL_s: = \{ x \in \hH^\CC \nr{ such that } [x,s] = 0 \},
\\
& P_s: = \{ g \in H^\CC \nr{ such that } e^{t s} g e^{-t s}  \nr{ is bounded as } \RR \ni t \to \infty \},
\\ 
& L_s: = \{ g \in H^\CC \nr{ such that } \Ad(g)(s) = s \}.
\end{align*}

The following properties hold:
\begin{enumerate}
 \item \label{it 1} Both $\pP_s$ and $\lL_s$ are Lie subalgebras of $\hH^\CC$ and $P_s$ and $L_s$ are connected subgroups of $H^\CC$. 
 \item \label{it 2} Let $\widetilde{\chi}$ be an antidominant character of $P_A$. Then there exists $s_\chi$ for which we have inclusions $\pP_A \subset \pP_{s_\chi}$, $\lL_A \subset \lL_{s_\chi}$, $P_A \subset P_{s_\chi}$ and $L_A \subset L_{s_\chi}$, with equality if $\chi$ is strictly antidominant. Furthermore $\chi$ is a strictly antidominant character of $\pP_{s_\chi}$.
 \item \label{it 3} For any $s \in i \hH$ there exists $h \in H$ and a standard parabolic subgroup $P_A$ such that $P_s = h P_A h^{-1}$ and $L_s = hL_A h^{-1}$. Furthermore, there is a strictly antidominant character $\chi$ of $P_A$ such that $s = h s_\chi h^{-1}$. 
\end{enumerate}
\end{lemma}

\begin{proof}
This result is contained in \cite[Lemma 2.5]{oscar&ignasi&gothen}, although we provide a proof for the sake of completeness.

One has from the definitions that $\pP_s$ and $\lL_s$ are subalgebras and $P_s$ and $L_s$ groups. Take $T_s$ to be the closure of $\{ e^{i ts} \, | \, t \in \RR \}$. Then $L_s$ is the centralizer of $Z_{H^\CC}(T_s)$ so it is connected by \cite[Th. 13.2]{borel1}. To prove that $P_s$ is connected note that, if $g \in P_s$ ({\it i.e.} $e^{ts} g e^{-ts}$ bounded as $t \to \infty$) then the limit exists, and we denote it by $\pi_s(g)$. Since it is a limit, it follows that $\pi_s(g) \in L_s$. This gives a morphism of Lie groups $\pi_s : P_s \to L_s$ that can be identified with the projection $P_s \to P_s/U_s \cong L_s$, where
\[
U_s := \{ g \in H^\CC \text{ such that } e^{ts} g e^{-ts} \text{ converges to } 1 \text{ as } t \to \infty \} \subset P_S
\]
is the unipotent radical of $P_s$. Then, for every $g \in P_s$, the map $\gamma : [0, \infty ) \to H^\CC$, defined as $\gamma(t) = e^{ts} g e^{-ts}$, extends to give a path from $g$ to $L_s$. Since $L_s$ is connected, it follows that $P_s$ is connected as well. This proves the first statement.

Let $\chi = \delta + \Sigma n_\alpha \lambda_\alpha$ be an antidominant character of $P_A$. Let $\beta = \Sigma m_\alpha \alpha$ be a root and take $u \in \hH_\alpha$. One has $[s_{\chi}, u] = \langle s_{\chi}, \beta \rangle u = \langle \chi, \beta \rangle u = \left ( \Sigma m_\alpha n_\alpha \langle \alpha, \alpha \rangle / 2 \right ) u$. Hence $\Ad(e^{t s_\chi})(u) = \left ( \Sigma \exp(t m_\alpha n_\alpha \langle \alpha, \alpha \rangle /2) \right ) u$, so this remains bounded as $t\to \infty$ if $m_\alpha \geq 0$ for any $\alpha$ such that $n_\alpha \leq 0$. This implies that $\pP_A \subset \pP_s$ and $\lL_A \subset \lL_s$, the inclusions being equalities when $\chi$ is strictly antidominant. The analogous results for $P_A \subset P_s$ and $L_A \subset L_s$ follow from this and the fact that they are connected. This finishes the proof of the second statement. 

To prove the third statement take a maximal torus $T_s$ containing $\{ e^{i ts} \, | \, t \in \RR \}$ and choose $h \in H$ such that $h^{-1} T_s h = T$ and $\Ad(h^{-1})(s)$ belongs to the Weyl chamber in $\tT$ corresponding to the choice of $\Delta(\hH^\CC, \tT^\CC)$. The proof follows from \eqref{it 2}.
\end{proof}


\begin{lemma}
\label{lm chi circ theta is antidominant}
Take a maximally compact $\theta$-stable Cartan subalgebra and a lexicographic order as in Remark \ref{rm Delta preserved by theta}. Let $\pP$ be a standard parabolic subalgebra of $\hH^\CC$ and let $\chi$ be an antidominant character of $\pP$. If $\pP_A$ is preserved by $\theta$, then $\theta\chi := \chi \circ \theta$ is an antidominant character of $\pP_A$. 
\end{lemma} 
 
\begin{proof}
Note that in the context of Remark \ref{rm Delta preserved by theta}, $\theta$ preserves the set of simple roots $\Delta(\hH^\CC, \tT^\CC)$. If $\pP_A$ is preserved by $\theta$, then $\theta(A) \subset A$. Then it is trivial to see that $\theta \chi$ is antidominant as well.
\end{proof}






\subsection{$G$-Higgs bundles on Riemann surfaces}

Let $\Sigma$ be a compact Riemann surface and denote by $\Omega^1_\Sigma$ its canonical bundle. Let $G$ be a connected real form of the complex semisimple Lie group $G^\CC$. Let $H \subset G$ be a maximal compact subgroup. Note that $H^\CC$ is a connected complex reductive Lie group.

A {\it $G$-Higgs bundle} over $\Sigma$ is a pair $(E,\Phi)$ where $E$ is a holomorphic $H^\CC$-bundle over $\Sigma$ and $\Phi$, called the {\it Higgs field}, is a holomorphic section of $E(\mM^\CC) \otimes \Omega^1_\Sigma$, where $E(\mM^\CC)$ is the vector bundle associated to the isotropy representation. Two $G$-Higgs bundles $(E,\Phi)$ and $(E',\Phi')$ are {\it isomorphic} if there exists an isomorphism of $H^\CC$-bundles $f: E \to E'$ such that $(f \otimes \id)^*\Phi' = \Phi$.


\begin{remark} \label{rm MmH is MH^CC}
The Cartan decomposition of the Lie algebra $\hH$ of a compact Lie group $H$ is $\hH = \hH \oplus 0$, so $\mM^\CC = 0$. It follows that a $H$-Higgs bundle is the same thing as a principal $H^\CC$-bundle..
\end{remark}

\begin{remark}
The Cartan decomposition of the Lie algebra $\hH^\CC$ of a complex reductive Lie group $H^\CC$ is $\hH^\CC = \hH \oplus i \hH$, so $\mM^\CC = \hH^\CC$. A $H^\CC$-Higgs bundle is a pair $(E,\Phi)$ where $E$ is a principal $H^\CC$-bundle and $\Phi \in H^0(\Sigma,E(\hH^\CC) \otimes \Omega^1_\Sigma)$.
\end{remark}

Let $S$ be an affine scheme $S$ and denote by $p : \Sigma \times S \to \Sigma$ the natural projection. We say that an {\it $S$-family} of $G$-Higgs bundle is a pair $(E_S, \Phi_S)$, where $E_S$ is a principal $H^\CC$-bundle over $\Sigma \times S$ (i.e. an $S$-family of principal $H^\CC$-bundles) and $\Phi_S$ is an element of $H^0(\Sigma \times S, E_S(\mM^\CC) \otimes p^*\Omega^1_\Sigma)$. Two $S$-families of $G$-Higgs bundles $(E_S,\Phi_S)$ and $(E'_S,\Phi'_S)$ are {\it isomorphic} if there exists an isomorphism of $H^\CC$-bundles $f_S: E_S \to E'_S$ such that $(f_S \otimes \id)^*\Phi_S' = \Phi_S$.

As is well known, in order to define a good moduli problem for the classification of $G$-Higgs bundles, one needs to introduce the notion of semistability.

The Killing form on $\lie{g}$ induces a Hermitian structure on $\lie{m}^\CC$ which is preserved by the action of $H^\CC$. This allows us to define the complex subspace
\begin{equation} \label{eq definition of mM^-}
(\mM^\CC)^{-}_{\chi} := \{ x \in \mM^\CC \nr{ such that } \iota(e^{t s_{\chi}} ) x \nr{ remains bounded as } \RR \ni t \to \infty \}.
\end{equation}
Let $E$ be a holomorphic $H^\CC$-bundle and $\sigma$ a holomorphic section of $E(H^\CC/P_A )$, i.e. a reduction of the structure group giving the $P_A$-bundle $E_\sigma$. We see that $(\mM^\CC)^-_\chi$ is invariant under the action of $P_{s_\chi}$ and by Lemma \ref{lm pPA in pP- y lLa in lL-} we have that $P_A \subset P_{s_\chi}$. Then we define 
\[
E(\mM^\CC)^{-}_{\sigma,\chi} := E_\sigma \times_{P_A} (\mM^\CC)^{-}_\chi.
\]

Now define
\[
(\mM^\CC)^0_\chi := \left \{ x \in \mM^\CC \nr{ such that } [s_\chi, x] = 0 \right \}. 
\]
This subspace is invariant under $L_{s_\chi}$ and hence under $L_A$ by Lemma \ref{lm pPA in pP- y lLa in lL-}. Suppose that $\sigma_L$ is a reduction of the structure group of $E_\sigma$ giving the $L_A$-bundle $E_{\sigma_L}$. Let us set
\[
E(\mM^\CC)^{0}_{\sigma_L,\chi} := E_{\sigma_L} \times_{L_A} (\mM^\CC)^{0}_\chi \subset E(\mM^\CC)^{-}_{\sigma,\chi}.
\]

Given a $H^\CC$-bundle $E$ with a reduction of the structure group $\sigma$ to the parabolic subgroup $P_A$ and an antidominant character $\chi$ of $\pP_A$, we define the degree of $E$ with respect to $\sigma$ and $\chi$ as in \cite[Section 5]{garcia-prada&peon&ramanan},
\[
\deg_{\sigma, \chi}(E) := \frac{1}{n}\deg \left ( \widetilde{\chi^n}_*E_\sigma \right ),
\]
where, following Remark \ref{rm exponentiation of characters}, $\chi^n$ is a character of $\pP_A$ that exponentiates to $P_A$.

We say that the $G$-Higgs bundle $(E,\Phi)$ is {\it semistable} (resp. {\it stable}) if for any parabolic subgroup $P_A \subset H^\CC$, any antidominant character $\chi$ of $\lie{p}_A$, and a reduction of the structure group $\sigma$ to the parabolic subgroup $P_A$ such that $\Phi \in H^0(\Sigma,E(\mM^\CC)^{-}_{\sigma,\chi} \otimes \Omega^1_\Sigma)$, we have
\[
\deg_{\sigma, \chi}(E) \geq 0 \qquad (\nr{resp. } \deg_{\sigma, \chi}(E) > 0). 
\]
Also, we say that $(E,\Phi)$ is {\it polystable} if it is semistable and for any $P_A$, $\chi$ and $\sigma$ as above, such that
$\Phi \in H^0(X,E(\mM^\CC)^{-}_{\sigma,\chi})$, $P_A \neq H^\CC$ and $\chi$ is strictly antidominant, and such that
\[
\deg_{\sigma, \chi}(E) = 0,
\]
there is a holomorphic reduction of the structure group $\sigma_L$ to the associated Levi subgroup $L_A$ and $\Phi$ is contained in $H^0(\Sigma,E(\mM^\CC)^{0}_{\sigma_L,\chi} \otimes \Omega^1_\Sigma)$.

\begin{remark}
A principal $H^\CC$-bundle $E$ is semistable or polystable if the $H$-Higgs bundle $(E,0)$ is respectively semistable or polystable.
\end{remark}

Given a semisimple subgroup $L \subset G$ preserved by the Cartan decomposition, its Lie algebra $\lL$ decomposes into $\lL_{\hH} \oplus \lL_{\mM}$, where $\lL_{\hH} = \lL \cap \hH$ and $\lL_{\mM} = \lL \cap \mM$. The subgroup $L_H = L \cap H$ is the maximal compact subgroup of $L$. We say that {\it $(E,\Phi)$ reduces to $L$} if there exists a reduction of structure group $\sigma$ of $E$ to $L_H^\CC$, giving the $L_H^\CC$-bundle $E_\rho$ and $\Phi(E_\sigma(\lL_{\mM}^\CC)) \subset E_\sigma(\lL_{\mM}^\CC) \otimes \Omega^1_\Sigma$.




Once we have defined the notion of semistability and polystability, it is possible to construct the moduli functor for the classification problem of $G$-Higgs bundles,
\begin{equation} \label{eq usual moduli functor}
\morph{\Mmm_\Sigma(G)}{\Aff}{\Sets}{S}{\left\{
	   \begin{array}{l}
		 \textnormal{Isomorphy classes of $S$-families} \\
	     \textnormal{of semistable $G$-Higgs bundles, } \\
	     \textnormal{with trivial characteristic class.}
	       \end{array}
	     \right\} .}{}{}
\end{equation}

We denote by $\Mm_\Sigma(G)$ the {\em moduli space of $G$-Higgs bundles}. Its existence follows, in full generality, from the work of Schmitt \cite{schmitt}.

\begin{theorem}[\cite{schmitt} Theorem 2.8.1.2] \label{tm MmG exists}
There exists a scheme $\Mm_\Sigma(G)$ corepresenting the moduli functor $\Mmm_\Sigma(G)$. The points of $\Mm_\Sigma(G)$ correspond to isomorphism classes of polystable $G$-Higgs bundles.
\end{theorem}


\begin{remark}
We have not given any formal definition of $S$-equivalence for $G$-Higgs bundles. This is done in \cite{schmitt} and also in \cite[Section 2.10]{oscar&ignasi&gothen}, where Jordan--H\"older filtrations are defined. For our purposes, it is sufficient to say that two semistable $G$-Higgs bundles are $S$-equivalent if they determine the same point of the moduli space $\Mm_\Sigma(G)$.
\end{remark}

Following Simpson \cite{simpson1, simpson2}, one can give a rigidification of the moduli functor that provides a fine moduli space. For a fixed geometric point $x_0 \in \Sigma$, we define a {\it framing} of the $H^\CC$-bundle $E$ to be an isomorphism $\xi: E|_{\{ x_0 \}} \stackrel{\cong}{\lra} H^\CC$. Given an $S$-family $(E_S, \Phi_S)$ of $G$-Higgs bundles, we say that a framing for the family is an isomorphism $\xi_S: E|_{\{ x_0 \} \times S} \stackrel{\cong}{\lra}  H^\CC \otimes \Oo_S$. Two ($S$-families of) $G$-Higgs bundles with framing, $(E_S, \Phi_S, \xi_S)$ and $(E'_S, \Phi'_S, \xi'_S)$, are isomorphic if there exists an isomorphism of ($S$-families of) $G$-Higgs bundles $f : (E_S, \Phi_S) \stackrel{\cong}{\longrightarrow} (E'_S, \Phi'_S)$ such that $\xi_S = \xi'_S \circ f|_{\{ x_0 \} \times S}$. Let us define as follows the moduli functor for the classification of $G$-Higgs bundles with framing,
\[
\morph{\Fff_\Sigma(G,x_0)}{\Aff}{\Sets}{S}{\left\{
	   \begin{array}{l}
		  \textnormal{Isomorphy classes of $S$-families} \\
	     \textnormal{of semistable $G$-Higgs bundles} \\
	     \textnormal{with framing at $x_0$ and trivial} \\
	     \textnormal{characteristic class.} 
	       \end{array}
	     \right\}.}{}{}
\]

\begin{proposition}[\cite{simpson2} Theorem 9.6 and Proposition 9.7 for the case $G = H^\CC$] \label{pr RG}
There exists a scheme $\Ff_\Sigma(G,x_0)$ representing the functor $\Fff_\Sigma(G,x_0)$. 
Furthermore, there exists an $H^\CC$-action on $\Ff_\Sigma(G,x_0)$ and
\begin{equation} \label{eq MmG = RG quotiented by H^C}
\Mm_\Sigma(G) \cong \git{\Ff_\Sigma(G,x_0)}{H^\CC.}
\end{equation}
The closed orbits of \eqref{eq MmG = RG quotiented by H^C} are those given by $G$-Higgs bundles with framing whose underlying $G$-Higgs bundles are  polystable.
\end{proposition}

\begin{proof}
With minor changes, we can use \cite[Theorem 3.13]{franco&garcia-prada&newstead_2} to extend \cite[Theorem 9.6 and Proposition 9.7]{simpson2} to the case of a real semisimple group $G$.
\end{proof}

In agreement with Simpson \cite{simpson1, simpson2}, we refer to the scheme $\Ff_\Sigma(G, x_0)$ as the {\it representation space} of $G$-Higgs bundles.

\subsection{Principal bundles over elliptic curves}
\label{sc G-bundles}

From now on, $(X, x_0)$ (or just $X$) will denote an elliptic curve. We write $\hat{X}$ for the variety $\Pic^0(X)$. The Abel--Jacobi map $x\mapsto\Oo(x)\otimes\Oo(x_0)^{-1}$ gives an isomorphism $X \cong \hat{X}$, which induces an abelian group structure on $X$. Having in mind the isomorphism $X \cong \hat{X}$, we will maintain the use of $\hat{X}$ through this section in order to clarify the exposition.



Let $\rho : \pi_1(X) \to H$ be a representation of the fundamental group into a compact Lie group $H$, we shall refer to such a representation as a {\it unitary representation}. Let $\ol{\rho} : \pi_1(X) \to H/Z_H(H)$ be the induced representation. We say that $\rho$ is {\it topologically trivial} if $\ol{\rho}$ can be lifted to a representation into the universal cover of $H/Z_H(H)$.
Given a topologically trivial representation of the fundamental group $\rho$, one can define a holomorphic $H^\CC$-bundle that we denote by $E_\rho$. If $\wt{X} \to X$ is the $\pi_1(X)$-bundle defined by the universal cover of $X$, we define $E_\rho := \rho_*(\wt{X})$, where $\rho_*$ denotes the extension of structure group associated to the representation. 
This construction is shown in \cite[Section 6]{atiyah&bott} and \cite{ramanathan_stable} where is also stated that
\begin{itemize}
\item the bundle $E_{\rho}$ is polystable,
\item $E_\rho$ has trivial characteristic class,
\item two bundles $E_{\rho_1}$ and $E_{\rho_2}$ are isomorphic if and only if $\rho_1$ and $\rho_2$ are conjugate, and
\item every polystable $H^\CC$-bundle is isomorphic to $E_{\rho}$ for some topologically trivial unitary representation $\rho : \pi_1(X) \to H$.
\end{itemize} 


For every element $y \in \zZ_{\hH}(\rho)$, one can define
\[
\morph{\eta_y}{\CC}{Z_{H^\CC}(\rho)}{t}{\exp_H( y\cdot t).}{}
\]
Using $\eta_y$, one can define 
\[
(\eta_y)_* : H^1(X, \Oo_X) \longrightarrow H^1(X, Z_{H^\CC}(\rho)),
\]
and we construct the $Z_{H^\CC}(\rho)$-bundle
\[
L_y := (\eta_y)_*\xi,
\]
where $\xi$ is a fixed non-zero element of the 1-dimensional space $H^1(X,\Oo_X)$. Due to the commutativity, the following map is a morphism of groups,
\[
\morph{\mu_\rho}{\im(\rho) \times Z_{H^\CC}(\rho)}{H^\CC}{(a,b)}{ab.}{}
\]
Note that $E_\rho$ is naturally a $\im(\rho)$-bundle while $L_y$ is a $Z_{H^\CC}(\rho)$-bundle and therefore we can consider the associated extension of structure groups giving a $H^\CC$-bundle,
\begin{equation} \label{}
\label{eq definition of E_rho y}
E_{\rho,y} := (\mu_\rho)_*(E_\rho \times_X L_y).
\end{equation}
Fix a Stein cover $\{ U_{i} \}$ on $X$ such that each $U_i$ is simply connected and each $U_i \cap U_j$ is either connected or empty for all $i \neq j$. Let $\{ f_{ij} \}$ be a non-zero $1$-cocycle for the cover $\{ U_i \}$ associated to $\xi \in  H^1(X,\Oo_X)$. Let $\{ h_{ij} \}$ be the normalized transition functions of $E_\rho$ with respect to the open cover $\{ U_i \}$. Then, the transition functions of $E_{\rho,y}$ for the cover $\{ U_i \}$ are
\begin{equation} \label{eq normalized transition functions}
\{ h_{ij} \exp_H(f_{ij} y) \}.
\end{equation}
Since $y$ and $\rho$ commute, one can check that \eqref{eq normalized transition functions} satisfies the cocycle condition.

\begin{remark}
\label{rm E_rho and E_rho y are S-equivalent}
Note that when $y \in \zZ_{\hH}(\rho)$ is nilpotent, $E_{\rho,y}$ is semistable and S-equivalent to $E_\rho$.
\end{remark}

\begin{remark}
\label{rm set of all possible Higgs fields}
Let $V$ be a complex vector space on which $H^\CC$ acts. Then, for any $E_{\rho,y}$,
\[
H^0(X,E_{\rho,y}(V)) = Z_V(\im \rho) \cap Z_V(y).
\]
\end{remark}

Over an elliptic curve, every semistable $H^\CC$-bundle can be expressed in these terms.

\begin{proposition}[\cite{friedman&morgan} Theorem 3.6 and Theorem 4.1]
\label{pr normalized transition functons}
\ 
\begin{enumerate}
\item Let $E$ be a semistable principal $H^\CC$-bundle over the elliptic curve $X$. Then there exist a central unitary representation $\rho$ and a nilpotent element $y \in \zZ_{\hH^\CC}(\rho)$ such that $E \cong E_{\rho, y}$. 
\item The group of automorphisms of $E_{\rho,y}$ is identified with 
\[
\Aut_{H^\CC}(E_{\rho,y}) = Z_{H^\CC}(\rho,y) = Z_{H^\CC}(\rho) \cap Z_{H^\CC}(y).
\]
\item $E_{\rho,y}$ and $E_{\rho',y'}$ are isomorphic if and only if $\rho$ and $\rho'$ are conjugate by an element $h \in H$ sending $y$ to $y'$.
\end{enumerate}
\end{proposition}

The fundamental group of an elliptic curve is abelian, $\pi_1(X) \cong \ZZ
\times \ZZ$. Then, the representations associated to a polystable
$H^\CC$-bundle of trivial characteristic class are completely determined by
commuting pairs $(a,b) = (\rho(\alpha), \rho(\beta))$ where $[a,b] = \id$ and
such that the projections $\ol{a},\ol{b}$ of $a$ and $b$ to $H/Z_H(H)$ can be
lifted to a commuting pair in the universal cover of this last group. By a
result of Borel \cite{borel2}, such $a$ and $b$ are contained in the same
maximal torus $T \subset H$ (up to conjugation by $H$). From the
Narasimhan--Seshadri--Ramanathan Theorem \cite{narasimhan&seshadri,ramanathan_stable}, one obtains the following result.

\begin{proposition}[\cite{friedman&morgan&witten}, \cite{laszlo}]
\label{co if deg E then E reduces to T}
Let $H$ be a compact group and let $T$ be a maximal torus. Every topologically trivial polystable $H^\CC$-bundle over the elliptic curve $X$ admits a reduction of structure group to $T$. 
\end{proposition}

Given a torus $T$, denote its {\it cocharacter lattice} by
\begin{equation}
\label{eq Lambda_S}
\Lambda_T := \Hom(\U(1), T) = \Hom(\CC^*, T^\CC),
\end{equation}
which is a lattice in $\tT$. Note that the fundamental group is $\pi_1(T) = \Lambda_T$. One has the natural isomorphism of groups
\begin{equation} \label{eq definition of Psi}
\map{\CC^* \otimes_\ZZ \Lambda_T}{T^\CC}{\sum_{i} u_i \otimes_\ZZ \lambda_i}{\Pi_i \lambda_{i}(u_i).}{\cong}
\end{equation}

Take the Poincar\'e bundle $\Pp_{\CC^*} \to X \times \hat{X}$. For a given torus $T$, using the isomorphism \eqref{eq definition of Psi} and fibre products of the Poincar\'e bundle, one can construct a family of $T^\CC$-bundles with trivial characteristic class,
\begin{equation} \label{eq definition of Pp_T}
\Pp_T \lra X \times (\hat{X} \otimes_\ZZ \Lambda_{T}).
\end{equation}
By \cite[Theorem 9.6]{simpson2} (among other references), $\Pp$ is a universal
family for the classification problem for $T^\CC$-bundles of characteristic
class $0$ with framing. 

Recall Proposition \ref{co if deg E then E reduces to T}. Let $i : T^\CC \hookrightarrow H^\CC$ be the natural injection and denote by $i_*$ the extension of structure group associated to it. As a consequence, the family 
\begin{equation} \label{eq definition of Ee_0}
\Ee_H : = i_*(\Pp_T) \to X \times (\hat{X} \otimes_\ZZ \Lambda_T)
\end{equation}
induces a surjective morphism from its parametrizing space to the moduli space $M_X(H^\CC)$ of topologically trivial $H^\CC$-bundles: 
\[
\hat{X} \otimes_\ZZ \Lambda_T \to M_X(H^\CC).
\]
There is a standard action of the Weyl group $W = W(H^\CC,T^\CC) = N_H(\tT)/Z_H(\tT)$ on $\Lambda_T$ which extends naturally to an action on $\hat{X} \otimes_\ZZ \Lambda_T$. The previous surjection factors through this action giving a bijection. Since the moduli space $M_X(H^\CC)$ is a normal variety, this is enough to prove the following.

\begin{theorem}[\cite{friedman&morgan&witten} Theorem 2.6, \cite{laszlo} Theorem 4.16] \label{tm MmH}
Let $H^\CC$ be a connected complex reductive Lie group and let $T \subset H$ be a maximal torus. Then
\begin{equation} \label{eq MmH}
M_X(H^\CC) \cong \quotient{(\hat{X} \otimes_\ZZ \Lambda_T)}{W}.
\end{equation}
\end{theorem}

\section{$G$-Higgs bundles over elliptic curves}
\label{sc G-Higgs bundles on EC}

Over an elliptic curve $X$, one has $\Omega^1_X \cong \Oo_X$. Therefore, a $G$-Higgs bundle over $X$ is a pair $(E,\Phi)$, where $E$ is a principal holomorphic $H^\CC$-bundle and $\Phi \in H^0(X,E(\mM^\CC))$.

\subsection{Stability in terms of the underlying principal bundle}
\label{sc stability}

We have the non-canonical isomorphism $H^0(X,\Oo_X) \cong \CC$. To simplify the presentation of the results in this section, we pick (non-canonically) a non-zero element of this space, $\s \in H^0(X,\Oo_X)$.

\begin{proposition} \label{pr E Phi semistable implies E semistable}
Let $(E,\Phi)$ be a semistable $G$-Higgs bundle. Then $E$ is a semistable $H^\CC$-bundle. 
\end{proposition}

\begin{proof}
Fix a maximally compact $\theta$-stable Cartan subalgebra $\cC_0$ and a lexicographic order as in Remark \ref{rm Delta preserved by theta}.

Suppose that $E$ is an unstable $H^\CC$-bundle.
We know by the Harder-Nara\-simhan Theorem \cite[Section 10]{atiyah&bott} that $E$ has a reduction $\sigma$ to some parabolic subgroup $P_{\HN} \subset H^\CC$ giving the $P_{\HN}$-bundle $E_\sigma$. Since $P_{\HN}$ is defined up to conjugation, one can assume that it is a standard parabolic subgroup associated to the subset $A \subset \Delta(\hH^\CC,\tT^\CC)$.

We know that $E_{G^\CC}$ is an unstable $G^\CC$-bundle and so we can apply again the Harder--Narasimhan Theorem to obtain a reduction $\gamma$ to the parabolic subgroup $Q_{\HN} \subset G^\CC$ giving the $Q_{\HN}$-bundle $(E_{G^\CC})_\gamma$. We take $Q_{\HN}$ to be a standard parabolic. The theorem also ensures the existence of an antidominant character $\tau$ of $Q_{\HN}$ such that $\deg{\gamma,\tau}(E_{G^\CC}) < 0$, and implies that the holomorphic sections of the adjoint bundle are contained in the reductions to the Harder--Narasimhan parabolics
\[
H^0(X,E(\hH^\CC)) = H^0(X,E_\sigma(\pP_{\HN}))
\]
and
\[
H^0(X,E_{G^\CC}(\gG^\CC)) = H^0(X,(E_{G^\CC})_\gamma(\qQ_{\HN})).
\]
By \cite[Proposition 10.4]{atiyah&bott}, the Harder--Narasimhan reduction is functorial with respect to group homomorphisms, so $H^\CC \hookrightarrow G^\CC$ implies that $P_{\HN} \subset Q_{\HN}$ and therefore the Lie algebra $\qQ_{\HN}$ is preserved by the Cartan involution. As a consequence
\begin{equation}
\label{eq H0 E mMC = H0 EP qQ cap mMC}
H^0(X,E(\mM^\CC)) = H^0(X,E_\sigma(\qQ_{\HN} \cap \mM^\CC)).
\end{equation}

Recall that $\theta$ denotes the Cartan involution. Since $\qQ_{\HN}$ is preserved by $\theta$, by Lemma \ref{lm chi circ theta is antidominant} we know that there exists an antidominant character $\eta = \frac{1}{2}(\tau + \tau \circ \theta)$ of $\qQ_{\HN}$. Therefore, one has an antidominant character $\chi$ of $\pP_{\HN}$ such that $\chi = \tau|_{\hH^\CC}$ and then the representatives via the Killing form of $\chi$ and $\eta$ are equal, $s_\chi = s_\eta$. 

Since $\tau$ restricted to $\pP_{\HN}$ is equal to our character $\chi$ and $\gamma_*(E_{G^\CC}) = (\sigma_*E)_{Q_{\HN}}$, one has that 
\begin{equation}
\label{eq deg E sigma chi < 0}
\deg_{\sigma,\chi}(E) = \deg_{\gamma,\tau}(E_{G^\CC}) < 0.
\end{equation}

Recall from \eqref{eq definition of mM^-} the linear subspaces $(\mM^\CC)^{-}_{\chi}$ and $\qQ^{-}_{\eta}$. Since $s_\chi = s_\eta$ we know that $(\mM^\CC)^{-}_{\chi} = \qQ^{-}_\eta \cap \mM^\CC$. By Lemma \ref{lm pPA in pP- y lLa in lL-} we have $(\qQ_{\HN}) \subset \qQ_{\eta}^{-}$, so $\qQ_{\HN} \cap \mM^\CC \subset (\mM^\CC)^{-}_{\chi}$. The parabolic subgroup $P$ acts on both subalgebras so
\[
E_\sigma(\qQ_{\HN} \cap \mM^\CC) \subseteq E_\sigma((\mM^\CC)^{-}_{\chi}).
\]
Due to \eqref{eq H0 E mMC = H0 EP qQ cap mMC} and the statement above, we have
\begin{equation}
\label{eq H0 EPmM-chi = H0 E mMC}
H^0(X,E((\mM^\CC)^{-}_{\sigma, \chi})) = H^0(X,E(\mM^\CC)).
\end{equation}

The existence of an antidominant character $\chi$ of $\pP$ satisfying \eqref{eq deg E sigma chi < 0} and \eqref{eq H0 EPmM-chi = H0 E mMC} implies that every $G$-Higgs bundle of the form $(E,\Phi)$ is unstable.
\end{proof}

From Propositions \ref{pr E Phi semistable implies E semistable} and \ref{pr normalized transition functons} and Remark \ref{rm set of all possible Higgs fields}, one has the following description of semistable $G$-Higgs bundles up to isomorphism.

\begin{corollary} \label{co description of semistable G-Higgs bundles}
Let $X$ be an elliptic curve. 
\begin{enumerate}

\item Every semistable $G$-Higgs bundle over $X$ is isomorphic to $(E_{\rho,y}, z \otimes \s)$ for some topologically trivial unitary representation $\rho : \pi_1(X) \to H$, $y \in \zZ_{\hH^\CC}(\rho)$ nilpotent, $z \in \zZ_{\mM^\CC}(\rho) \cap \zZ_{\mM^\CC}(y)$.

\item If $\rho$ is a topologically trivial unitary representation of $\pi_1(X)$, $y$ a nilpotent element of $\zZ_{\hH^\CC}(\rho)$ and $z \in \zZ_{\mM^\CC}(\rho) \cap \zZ_{\mM^\CC}(y)$, then the $G$-Higgs bundle $(E_{\rho, y}, z \otimes \s)$ is semistable.

\item The group of automorphisms of $(E_{\rho,y}, z \otimes \s)$ is identified with 
\[
\Aut_G(E_{\rho,y}, z \otimes \s) = Z_{H^\CC}(\rho,y,z) = Z_{H^\CC}(\rho) \cap Z_{H^\CC}(y) \cap Z_{H^\CC}(z).
\]

\item The $G$-Higgs bundles $(E_{\rho,y}, z \otimes \s)$ and $(E_{\rho',y'}, z' \otimes \s)$ are isomorphic if and only if $\rho$ and $\rho'$ are conjugate by an element $h \in H^\CC$ sending $y$ to $y'$ and $z$ to $z'$.
\end{enumerate}
\end{corollary}

We continue with our study of stability.

\begin{proposition} \label{pr E Phi polystable implies E polystable}
Let $(E,\Phi)$ be a polystable $G$-Higgs bundle. Then $E$ is a polystable $H^\CC$-bundle. 
\end{proposition}

\begin{proof}
Suppose that  $(E,\Phi)$ is a polystable $G$-Higgs bundle. By Corollary \ref{co description of semistable G-Higgs bundles} one can assume with no loss of generality that $(E,\Phi) = (E_{\rho, y}, z \otimes \s)$ where $y \in \zZ_{\hH^\CC}(\rho)$ is nilpotent. Since $z$ belongs to $\zZ_{\mM^\CC}(\rho) \cap \zZ_{\mM^\CC}(y)$ we can construct the semistable $G$-Higgs bundle $(E_\rho, z \otimes \s)$. Using Remark \ref{rm E_rho and E_rho y are S-equivalent}, we see that $(E_{\rho,y},z \otimes \s)$ and $(E_\rho, z \otimes \s)$ are S-equivalent.  

In each S-equivalence class, there is only one isomorphy class of polystable $G$-Higgs bundle. So, if $(E_{\rho}, z \otimes \s)$ is polystable as well, we would have that $(E_{\rho,y},z \otimes \s)$ and $(E_\rho, z \otimes \s)$ are necessarily isomorphic and the proof would be completed since, in that case
\[
E = E_{\rho,y} \cong E_\rho
\]
is a polystable $H^\CC$-bundle.

Take a parabolic subgroup $P$ and a strictly antidominant character $\chi$ such that $\im(\rho) \subset P$ (giving a reduction $\sigma$ of $E_\rho$ to $P$), $z \in (\mM^\CC)^{-}_{\sigma,\chi}$ and 
\[
\deg_{\sigma, \chi}(E_\rho) = 0.
\]
We claim that there exists a reduction of $E_{\rho}$ to $L$, the Levi factor of $P$ and $z \in (\mM^\CC)^0_{\chi}$. This implies that $(E_\rho, z \otimes \s)$ is polystable and therefore $(E_{\rho,y}, z \otimes \s) \cong (E_\rho, z \otimes \s)$, so the proof is follows from this claim.

Let us prove the polystability of $ (E_\rho,z \otimes \s)$. Take $\pP'$ to be the minimal parabolic subalgebra containing $y$ and the parabolic subalgebra $\pP = \Lie(P)$. Let $P'$ be the parabolic subgroup associated with $\pP'$ and let $\chi' : P' \to \CC$ be the antidominant character determined by $s_\chi$ as in Lemma \ref{lm pPA in pP- y lLa in lL-} (therefore we have $s_\chi = s_{\chi'}$). By construction $\im(\rho) \times U_y$ is contained in $P'$ (we take $U_y$ to be the unipotent group generated by $y$), so there is a reduction $\sigma'$ of the structure group of $E_{\rho,y}$ to $P'$. Note that we have $z \in (\mM^\CC)^{-}_{\sigma',\chi'}$ since $P \subset P'$ and $s_\chi = s_{\chi'}$.

Let $n$ be a positive integer such that $(\chi')^n$ exponentiates to a character of the group $\wt{(\chi')^n} : P' \to \CC^*$. By construction of $E_{\rho, y}$, one has that $\wt{(\chi')^n}_* E_{\rho,y} \cong \wt{\chi^n}_* E_{\rho} \otimes \wt{(\chi')^n}_* L_y$, where the transition functions of $\wt{(\chi')^n}_* L_{y}$ are $\{e^{n d \chi' (y) f_{ij}} \}$. This line bundle is topologically trivial since we can give a connected path $\gamma$ on the moduli space of line bundles connecting $\wt{(\chi')^n}_* L_{y}$ with the trivial bundle. We have 
\begin{align*}
\deg_{\sigma',\chi'}(E_{\rho,y}) =& \frac{1}{n} \deg \left ( \wt{(\chi')^n}_* E_{\rho,y} \right ) = 
\\
= & \frac{1}{n} \deg \left ( \wt{\chi^n}_* E_{\rho} \otimes \wt{(\chi')^n}_*L_{y} \right ) =
\\
= & \frac{1}{n} \deg \left ( \wt{\chi^n}_* E_{\rho}) + \frac{1}{n}\deg (\wt{(\chi')^n}_*L_{y} \right ) =
\\
= & \deg_{\sigma,\chi}(E_{\rho}) + 0 =
\\
=&  0.
\end{align*}

Since $(E_{\rho,y},z\otimes \s)$ is polystable, there is a reduction of structure group $\sigma'_L$ of $E_{\rho,y}$ to the Levi factor $L'$ of $P'$, and $z \in (\mM^\CC)^0_{\chi'}$. Since $s_\chi = s_{\chi'}$, this implies a reduction of $E_{\rho}$ to $L$, the Levi factor of $P$, and $z \in (\mM^\CC)^0_{\chi}$, and as a consequence, we obtain that $(E_\rho,z \otimes \s)$ is polystable. 
\end{proof} 

Using Proposition \ref{pr E Phi polystable implies E polystable}, one can give a first description of polystable $G$-Higgs bundles. This description will be incomplete until Corollary \ref{co complete description of polystable G-Higgs bundles}.

\begin{corollary} 
\label{co description of polystable G-Higgs bundles}
Let $X$ be an elliptic curve. Every polystable $G$-Higgs bundle with trivial characteristic class over $X$ is isomorphic to $(E_{\rho},  z \otimes \s)$ for some topologically trivial unitary representation $\rho : \pi_1(X) \to H$ and $z \in \zZ_{\mM^\CC}(\rho)$.
\end{corollary}

We know, from Corollary \ref{co description of polystable G-Higgs bundles}, that every polystable $G$-Higgs bundle is isomorphic to one of the form $(E_\rho, z \otimes \s)$. However, not every $G$-Higgs bundle of this form is polystable. The following result will characterize them. 

\begin{lemma} \label{lm complete description of polystable}
Let $\rho : \pi_1(X) \to H$ be a topologically trivial unitary representation and let $z$ be an element of $\zZ_{\mM^\CC}(\rho)$. The $G$-Higgs bundle $(E_\rho, z \otimes \s)$ is polystable if and only if $z$ lies in a maximal abelian subalgebra $\aA_{\rho}^\CC$ of $\zZ_{\mM^\CC}(\rho)$.
\end{lemma}

\begin{proof}
Recall that the Real Chevalley Theorem (see for instance \cite[Theorem 6.57]{knapp}) studies the GIT quotient $\zZ_{\mM^\CC}(\rho) \Slash Z_{H^\CC}(\rho)$, stating that the $Z_{H^\CC}(\rho)$-orbit of $z$ is closed if and only if $z$ is contained in some maximal abelian subalgebra $\aA_{\rho}^\CC \subset \zZ_{\mM^\CC}(\rho)$. 

Suppose that $z$ is not contained in any maximal abelian subalgebra $\aA_{\rho}^\CC$ of $\zZ_{\mM^\CC}(\rho)$. Then, by the Hilbert-Mumford Criterum, there exists a $1$-parameter subgroup $\lambda : \CC^* \to Z_{H^\CC}(\rho)$ such that $\lim_{t \to 0} \lambda(t) \cdot z $ exists but does not belong to the orbit. One can consider $\lambda$ to be a $1$-parameter subgroup of $H^\CC$ and we let $\lambda$ act on $\Ff_X(G,x_0)$. Since the image of $\lambda$ is contained in $Z_{H^\CC}(\rho)$ its action on $E_\rho$ is the identity. The previous discussion
implies, trivially, that
\[
\lim_{t \to 0} \lambda(t) \cdot (E_\rho,  z\otimes \s)  
\]
exists but does not belong to the $H^\CC$-orbit of $(E_\rho, z \otimes \s)$ inside $\Ff_X(G,x_0)$. Then, the $H^\CC$-orbit of $(E_\rho, z \otimes \s)$ is not closed and $(E_\rho, z \otimes \s)$ is not polystable by Proposition \ref{pr RG}.

Now, we suppose that there exists a maximal abelian subalgebra $\aA_{\rho}^\CC$ of $\zZ_{\mM^\CC}(\rho)$ containing $z$. Then, the subalgebra $\zZ_{\hH^\CC}(\rho, z) = \zZ_{\hH}(\rho, z)^\CC$ is reductive. Take an abelian subalgebra $\sS$ of $\zZ_\hH(\rho, z)$ and let $S \subset H$ be the torus with Lie algebra $\sS$. Note that $\im(\rho) \subset Z_{H^\CC}(S)$ and $z \in \zZ_{\mM^\CC}(\sS)$, by construction. The $G$-Higgs bundle $(E_\rho, z \otimes \s)$ reduces to a $Z_G(S)$-Higgs bundle. Recall the definitions in Lemma \ref{lm pPA in pP- y lLa in lL-} and \eqref{eq definition of mM^-} for any $s' \in i \zZ_{\hH}(S)$. Note that, if one has 
\[
\im(\rho) \subset P_{s'}
\]
and 
\[
z \in (\mM^\CC)^-_{s'},
\]
then, by the maximality of $\sS$ inside $\zZ_\hH(\rho, z)$, this implies that $s' \in i \sS$. Then, $(E_\rho, z \otimes \s)$ is a stable $Z_G(S)$-Higgs bundle and by \cite{bradlow&garcia-prada&mundet} it gives a solution of the Hitchin equations, so it is a polystable $G$-Higgs bundle.
\end{proof}

Using Lemma \ref{lm complete description of polystable} and Corollary \ref{co description of semistable G-Higgs bundles}, one can complete the description of polystable $G$-Higgs bundles that we started in Corollary \ref{co description of polystable G-Higgs bundles}.

\begin{corollary} \label{co complete description of polystable G-Higgs bundles}
Let $X$ be an elliptic curve.
\begin{enumerate}
\item Every polystable $G$-Higgs bundle over $X$ is isomorphic to $(E_{\rho},  z \otimes \s)$ for some topologically trivial unitary representation $\rho : \pi_1(X) \to H$ and $z \in \aA_{\rho}^\CC$, where $\aA_{\rho}^\CC$ is a maximal abelian subalgebra of $\zZ_{\mM^\CC}(\rho)$.

\item Let $\rho : \pi_1(X) \to H$ be a topologically trivial unitary representation, let $\aA_{\rho}^\CC$ be a maximal abelian subalgebra of $\zZ_{\mM^\CC}(\rho)$ and take $z \in \aA_{\rho}^\CC$. Every $G$-Higgs bundle of the form $(E_\rho, z \otimes \s)$ is polystable.

\item The group of automorphisms of $(E_{\rho}, z \otimes \s)$ is identified with 
\[
\Aut_G(E_{\rho}, z \otimes \s) = Z_{H^\CC}(\rho,z) = Z_{H^\CC}(\rho) \cap Z_{H^\CC}(z),
\]
and is a complex reductive subgroup of $H^\CC$.

\item The polystable $G$-Higgs bundles $(E_{\rho}, z \otimes \s)$ and $(E_{\rho'}, z' \otimes \s)$ are isomorphic if and only if $\rho$ and $\rho'$ are conjugate by an element $h \in H^\CC$ sending $z$ to $z'$.
\end{enumerate}
\end{corollary}

\subsection{The representation space}
\label{sc FfG}

Proposition \ref{pr E Phi semistable implies E semistable} allows us to describe $\Ff_X(G,x_0)$ in terms of $\Ff_X(H,x_0)$. Recall that $\Ff_X(H,x_0)$ is a fine moduli space and let $\Uu_H \to X \times \Ff_X(H,x_0)$ be the corresponding universal bundle. Take the obvious projection
\[
q : X \times \Ff_X(H,x_0) \lra \Ff_X(H,x_0).
\]
If $\Uu_H(\mM^\CC)$ is the vector bundle induced from $\Uu_H$ under the isotropy action of $H^\CC$ on $\mM^\CC$ and $R^1 q_* \Uu_H(\mM^\CC)$ the $1$-cohomology direct image sheaf under $q$. This is a sheaf over $\Ff_X(H,x_0)$ whose stalk over $(E,\xi)$ coincides with $H^1(X, E(\mM^\CC))$. Take the symmetric algera $\Sym^\bullet(R^1 q_* \Uu_H(\mM^\CC))$ associated to this sheaf and consider the scheme $\Spec \left (  \Sym^\bullet(R^1 q_* \Uu_H(\mM^\CC)) \right )$. Note that this scheme projects naturally to $\Ff_X(H,x_0)$,
\begin{equation} \label{eq Spec Sym Um to FfH}
p: \Spec \left (  \Sym^\bullet(R^1 q_* \Uu_H(\mM^\CC)) \right ) \twoheadrightarrow \Ff_X(H,x_0),
\end{equation}
and the fibre over $(E,\xi) \in \Ff_X(H,x_0)$ is $H^1(X,E(\mM^\CC))^*$.

\begin{proposition} \label{pr FfG onto FfH}
The scheme $\Ff_X(G,x_0)$ represents the moduli functor $\Fff_X(G,x_0)$ and one has an isomorphism of schemes
\begin{equation} \label{eq description of RG}
\Ff_X(G,x_0) \cong \Spec \left (  \Sym^\bullet(R^1 q_* \Uu_H(\mM^\CC)) \right ).
\end{equation}
Furthermore, the representation space of $G$-Higgs bundles projects to the representation space of $H^\CC$-bundles,
\begin{equation}
\label{eq RG to RH}
\map{\Ff_X(G,x_0)}{\Ff_X(H,x_0)}{(E,\Phi,\xi)}{(E,\xi),}{}
\end{equation}
and the fibre of \eqref{eq RG to RH} over $(E,\xi) \in \Ff_X(H,x_0)$ is $H^0(X,E(\mM^\CC))$. 
\end{proposition}

\begin{proof}
Recall the Cartan decomposition $\gG^\CC = \hH^\CC \oplus \mM^\CC$, where $\mM^\CC$ is orthogonal to $\hH^\CC$ under the Killing form. Since the adjoint bundle $E(\gG^\CC)$ is naturally self dual, this orthogonality implies that $E(\mM^\CC)$ is self-dual as well. Thanks to Serre duality and the triviality of the canonical bundle, one has a canonical identification 
\[
H^1(X,E(\mM^\CC))^* \cong H^0(X,E(\mM^\CC)).
\]
Let $\tau : X \times \Spec \left (  \Sym^\bullet(R^1 q_* \Uu_H(\mM^\CC)) \right )$ be the tautological section and take the family
\[
\Uu_G := ((\id \times p)^*\Uu_H, \tau) \to X \times \Spec \left (  \Sym^\bullet(R^1 q_* \Uu_H(\mM^\CC)) \right ). 
\]
It follows, by the universal properties of $\Uu_H$, that $\Uu_G$ is a universal family for the moduli functor $\Fff_X(G,x_0)$. Since $\Ff_X(G,x_0)$ corepresents this functor, one necessarily obtains the isomorphism \eqref{eq description of RG}.
\end{proof}

Let $\Ff_\Sigma(G,x_0)^{ps}$ denote the subset of $\Ff_\Sigma(G,x_0)$ given by the polystable $G$-Higgs bundles. In general, this subset is not open or closed inside $\Ff_\Sigma(G,x_0)$. The purpose of this section is to show that, in the case of an elliptic curve $\Sigma = X$, one can prove that $\Ff(G,x_0)^{ps} \subset \Ff_X(G, x_0)$ is closed.

Recall from \eqref{eq definition of Ee_0} the family of polystable $H^\CC$-bundles $\Ee_{H} \to X \times (\hat{X} \otimes_\ZZ \Lambda_T)$ and fix a framing $\xi$ at $x_0$ for it. The family $(\Ee_{H}, \xi) \to X \times (\hat{X} \otimes_\ZZ \Lambda_T)$ induces, by moduli theory, a morphism to the representation space 
\begin{equation} \label{eq definition of nu_0}
\nu_H : \hat{X} \otimes_\ZZ \Lambda_T \lra \Ff_X(H,x_0).
\end{equation}

\begin{lemma} \label{lm description of RH^ps}
Let $H$ be a compact Lie group. The polystable locus $\Ff_X(H,x_0)^{ps}$ is closed inside $\Ff_X(H,x_0)$. Furthermore, 
\begin{equation} \label{eq description of RH}
\Ff_X(H,x_0)^{ps} = H^\CC \cdot \nu_H( \hat{X} \otimes_\ZZ \Lambda_T),
\end{equation}
and
\begin{equation} \label{eq nu restricted to X otimes Lambda is an isomorphism}
\nu_H(\hat{X} \otimes_\ZZ \Lambda_T) \cong \hat{X} \otimes_\ZZ \Lambda_T.             \end{equation}
\end{lemma}

\begin{proof}
The map $\nu_H$ in \eqref{eq definition of nu_0} is closed since $\hat{X} \otimes_\ZZ \Lambda_T$ is compact. The image of $\nu_H$ is contained in $\Ff_X(H,x_0)^{ps}$. In fact $H^\CC \cdot \nu_H(\hat{X} \otimes_\ZZ \Lambda_T)$ is clearly contained in $\Ff_X(H,x_0)^{ps}$. Furthermore, since every polystable $H^\CC$-bundle is isomorphic to one parametrized by $\Ee_H$, one has that
\begin{equation} \label{eq H cdot X otimes Lambda surjects to MH^CC}
\git{H^\CC \cdot \nu_H(\hat{X} \otimes_\ZZ \Lambda_T)}{H^\CC} \lra M_X(H^\CC)
\end{equation}
is surjective and therefore an isomorphism since $H^\CC \cdot \nu_H(\hat{X} \otimes_\ZZ \Lambda_T)$ injects into $\Ff_X(H,x_0)$. This implies \eqref{eq description of RH} and therefore $\Ff_X(H,x_0)^{ps}$ is closed inside $\Ff_X(H,x_0)$.

Finally, recall that $M_X(H^\CC)$ is described in Theorem \ref{tm MmH} as the finite quotient \eqref{eq MmH}. Note that this, together with the surjection \eqref{eq H cdot X otimes Lambda surjects to MH^CC}, implies \eqref{eq nu restricted to X otimes Lambda is an isomorphism}.
\end{proof}

\begin{proposition} \label{pr RG^ps decomposes}
Let $H$ be a maximal compact subgroup of $G$ and let $\gG = \hH \oplus \mM$ be the Cartan decomposition of its Lie algebra. Then, the polystable locus $\Ff_X(G,x_0)^{ps}$ is closed inside $\Ff_X(G,x_0)$ and isomorphic to a closed subset of the direct product $\Ff_X(H,x_0)^{ps} \times \left ( \mM^\CC \otimes H^0(X, \Oo_X) \right )$.
\end{proposition}

\begin{proof}
Recall the family of polystable $H^\CC$-bundles $\Ee_{H} \to X \times (\hat{X} \otimes_\ZZ \Lambda_T)$ and take the projection 
\[
q : X \times (\hat{X} \otimes_\ZZ \Lambda_T) \longrightarrow \hat{X} \otimes_\ZZ \Lambda_T.
\]
Consider a construction analogous to \eqref{eq Spec Sym Um to FfH}, giving a natural projection
\[
\pi: \Sigma_G : = \Spec \left ( \Sym^\bullet (R^1 q_* \Ee_{H}(\mM^\CC) ) \right ) \twoheadrightarrow \hat{X} \otimes_\ZZ \Lambda_T.
\]
Note that the fibre over $t \in \hat{X} \otimes_\ZZ \Lambda_T$ is $H^0(X, \Ee_H|_t(\mM^\CC))$. Using the tautological section
\[
\tau : X \times \Sigma_G \longrightarrow \Ee_H(\mM^\CC),
\]
one can construct the family of $G$-Higgs bundles with framing 
\[
\left ( (\id \times \pi)^* \Ee_H, \tau, (\id \times \pi)^* \xi \right ) \longrightarrow X \times \Sigma_G.
\]
Thanks to Corollary \ref{co description of semistable G-Higgs bundles}, we know that this family parametrizes all polystable $G$-Higgs bundles of characteristic class $d$ (although, as we know by Lemma \ref{lm complete description of polystable}, some $G$-Higgs bundles parametrized by this family might be strictly semistable). By moduli theory, there exists a map from the parametrizing space of this family to the moduli space $\Ff_X(G,x_0)$,
\[
\nu'_G : \Sigma_G \longrightarrow \Ff_X(G,x_0).
\]

By Remark \ref{rm set of all possible Higgs fields} and the construction of the family $\Ee_H$, for every $t \in \hat{X} \otimes_\ZZ \Lambda_T$ one has
\begin{equation} \label{eq Phi in H^0 Ee mM}
H^0(X, \Ee_H|_t(\mM^\CC)) \subset H^0(X, \mM^\CC \otimes \Oo_X) \cong \mM^\CC \otimes H^0(X,\Oo_X).
\end{equation}
Indeed, the elements of $H^0(X, \Ee_H|_t(\mM^\CC))$ have the form $z \otimes \s$ where $z \in \mM^\CC$ commutes with the transition functions of $\Ee_H|_t = \Pp|_t \otimes E^{x_0}_{L,d}$. Note that the conjugation of $\Ee_H|_t$ by any $h \in H^\CC$ preserves the previous inclusion
\begin{equation} \label{eq ad_h H^0 Ee mM in mM}
H^0(X, \ad_h(\Ee_H|_t)(\mM^\CC)) \subset H^0(X, \mM^\CC \otimes \Oo_X).
\end{equation}

By \eqref{eq Phi in H^0 Ee mM}, one has that $\nu'_G(\Sigma_G)$ is isomorphic to a closed subset $\Ss_0$ of $\nu_H(\hat{X} \otimes_\ZZ \Lambda_T) \times (\mM^\CC \otimes H^0(X,\Oo_X))$. Furthermore, thanks to \eqref{eq ad_h H^0 Ee mM in mM}, one has that 
\begin{equation} \label{eq Ss inside Ff^ps}
H^\CC \cdot \nu'_G(\Sigma_G) \cong H^\CC \cdot \Ss_0,
\end{equation}
where $H^\CC \cdot \Ss_0$ is a closed subset of $\left ( H^\CC \cdot \nu_H(\hat{X} \otimes_\ZZ \Lambda_T) \right ) \times \left (\mM^\CC \otimes H^0(X,\Oo_X) \right )$. 

By \eqref{eq description of RH}, $H^\CC \cdot \Ss_0$ is a closed subset of $\Ff_X(H,x_0)^{ps} \times (\mM^\CC \otimes H^0(X,\Oo_X))$ and it corresponds to the restriction of $R^0q_*\Uu_H(\mM^\CC)$ to $\Ff_X(H,x_0)^{ps}$. Recalling that not every $G$-Higgs bundle parametrized by $H^\CC \cdot \Ss_0$ is polystable, we consider the closed subset
\[
\Ss := \left ( H^\CC \cdot \Ss_0 \right ) \cap \left ( \Ff_X(H,x_0)^{ps} \times (\mM^\CC)^{ss} \otimes H^0(X,\Oo_X) \right ), 
\]
where $(\mM^\CC)^{ss}$ is the closed subset of semisimple elements of $\mM^\CC$ given by
\[
(\mM^\CC)^{ss} := H^\CC \cdot \aA_{D}^\CC,
\]
with $\aA_{D}^\CC \subset \mM^\CC$ a maximal abelian subalgebra (recall that all the maximal abelian subalgebras of $\mM^\CC$ are conjugate under $H^\CC$). Thanks to Lemma \ref{lm complete description of polystable}, we see that $\Ss$ corresponds under \eqref{eq Ss inside Ff^ps} with those elements of $H^\CC \cdot \nu'_G(\Sigma_G)$ that are polystable, and this coincides with $\Ff_X(G,x_0)^{ps}$.
\end{proof}

\subsection{The moduli space}\label{section-moduli}

Using Proposition \ref{pr FfG onto FfH}, one can describe $\Mm_X(G)$ in terms of a fibration over $M_X(H^\CC)$.

\begin{corollary}
\label{pr MmG onto MG}
The moduli space of $G$-Higgs bundles projects onto the moduli space of $H^\CC$-bundles,
\begin{equation}
\label{eq MmG to MH}
\map{\Mm_X(G)}{M_X(H^\CC)}{(E,\Phi)}{E.}{}
\end{equation}
Let $\rho: \pi_1(X) \to H$ be a topologically trivial unitary representation and let $E_\rho$ be the polystable $H^\CC$-bundle associated to it. The fibre of the surjection \eqref{eq MmG to MH} over the isomorphism class of $E_\rho$ is identified with the vector space $\zZ_{\mM^\CC}(\rho) \Slash Z_{H^\CC}(\rho)$.
\end{corollary}

\begin{proof}

Since \eqref{eq RG to RH} is $H^\CC$-equivariant, it descends to 
\[
\Mm_X(G) \cong \git{\Ff_X(G,x_0)}{H^\CC} \lra M_X(H^\CC) \cong \git{\Ff_X(H,x_0)}{H^\CC}.
\]
The fibre over the isomorphism class of $E_\rho$ is $H^0(X,E_\rho(\mM^\CC)) \Slash \Aut_{H^\CC}(E_\rho)$. By Proposition \ref{pr normalized transition functons} and Remark \ref{rm set of all possible Higgs fields} this is identified with $\zZ_{\mM^\CC}(\rho) \Slash Z_{H^\CC}(\rho)$.
\end{proof}

Let $\theta$ be a Cartan involution of $\gG$ whose associated Cartan decomposition is $\gG = \hH \oplus \mM$ and fix a maximally compact $\theta$-stable Cartan subalgebra $\cC_0$ of $\gG$. We recall that $\cC_0 = \tT \oplus \aA_0$, where $\tT \subset \hH$ is the Lie algebra of a maximal torus $T$ of $H$ and $\aA_0 \subset \mM$. Fix once for all $x_\alpha \in (\gG^\CC)^\alpha$ for each $\alpha \in \Delta(\gG,\cC_0)$ and recall that every admissible root system $B$ defines as in \eqref{eq definition of cC_B} a Cartan subalgebra $\cC_B$ of $\gG$, that splits into $\tT_B \oplus \aA_B$ as we observe in \eqref{eq definition of tT_B} and \eqref{eq definition of aA_B}. Let us denote by $T_B \subset T$ the torus with Lie algebra $\tT_B$. Recall from \eqref{eq definition of Upsilon} that $\Upsilon$ is the set of all possible admissible systems.

\begin{lemma}
\label{pr E Phi reduces to a Cartan subgroup of G}
Let $(E,\Phi)$ be a polystable $G$-Higgs bundle over an elliptic curve. Then there exists an admissible root system $B \in \Upsilon$ such that $(E,\Phi)  \cong (E_\rho, z \otimes \s)$ where $\rho : \pi_1(X) \to T_B$ and $z \in \aA_B^\CC$.
\end{lemma}

\begin{proof}
Since $z$ is semisimple by Lemma \ref{lm complete description of polystable}, there exists a $\theta$-stable Cartan subalgebra $\cC$ that contains it. By construction $\im \rho \subset \exp(\cC \cap \hH)$ and $z \in (\cC \cap \mM)^\CC$. Finally, note that $\cC$ is conjugate to some $\cC_B$ by Lemma \ref{lm Cayley transform give all Cartan subalgebras}.
\end{proof} 

Every $\alpha \in I_{nc}(\gG,\cC_0)$ (the set of imaginary non-compact roots) is by definition an element of the dual space $\Hom(\tT,i\RR)$. Since $\Lambda_T = \Hom(\CC^*, T^\CC)$ can be seen as a lattice inside $\tT \subset \cC_0^\CC$ and $\alpha$ is a root of $R(\gG^\CC,\cC_0^\CC)$, we have that $\alpha(\Lambda_T) \subset i \ZZ$. Therefore, for each $\alpha$ one has a well defined projection
\[
\morph{\eta_\alpha}{\hat{X} \otimes_\ZZ \Lambda_T}{\hat{X}}{\sum_j L_j \otimes_\ZZ \lambda_j}{\bigotimes_j L_j^{\otimes \alpha(\lambda_j)i^{-1}}.}{}
\]
Given an admissible root system $B \in \Upsilon$ with $B = \{ \alpha_1, \dots, \alpha_\ell \}$ we define
\begin{equation}
\label{eq definition of X otimes Lambda restricted to B}
(\hat{X} \otimes_\ZZ \Lambda_T)|_B := \bigcap_{\alpha \in B} \eta_\alpha^{-1}(\Oo_X) 
\end{equation}
which is a closed subset of $\hat{X} \otimes_\ZZ \Lambda_T$. Note that $(\hat{X} \otimes_\ZZ \Lambda_T)|_0=\hat{X} \otimes_\ZZ \Lambda_T$.

\begin{remark} \label{rm Pp_T reduces to T_B}
Recall the definition of $\tT_B$ in \eqref{eq definition of tT_B} and the construction of $\Pp_T$ \eqref{eq definition of Pp_T}, where we made use of the isomorphism of groups \eqref{eq definition of Psi}. Note that, by the construction \eqref{eq definition of X otimes Lambda restricted to B}, the restriction of $\Pp_T$ to $(\hat{X} \otimes_\ZZ \Lambda_T)|_B$ reduces its structure group to the compact torus $T_B$.
\end{remark}

\begin{remark} \label{rm Ee_H0 reduces to T_B}
Let $B \in \Upsilon$ be an admissible root system and take $t \in (\hat{X} \otimes_\ZZ \Lambda_T)|_B$ and $z \in \aA_B^\CC$. By Remark \ref{rm Pp_T reduces to T_B} and Lemma \ref{lm complete description of polystable}, we have that $(\Ee_H|_{X \times \{ t \}}, z \otimes \s)$ is a polystable $G$-bundle of trivial characteristic class.
\end{remark}

\begin{remark} \label{rm description of X otimes Lambda_B}
Since $T_B$ is a subtorus of $T$, one has that the cocharacter lattice $\Lambda_{T_B} = \Hom(\CC^*, T_B^\CC)$ is a sublattice of $\Lambda_{T} = \Hom(\CC^*, T^\CC)$. We can describe
\[
(\hat{X} \otimes_\ZZ \Lambda_T)|_B = \hat{X} \otimes_\ZZ \Lambda_{T_B} \subset \hat{X} \otimes_\ZZ \Lambda_{T}.
\]
As a consequence, we observe that the $(\hat{X} \otimes_\ZZ \Lambda_T)|_B$ are irreducible. 
\end{remark}


 e has that $(E_\rho, z \otimes 1_\Oo)$ is a well defined Higgs bundle. It is polystable since it is polystable as a $G^\CC$-Higgs bundle (because $(E_\rho, z \otimes \s)$ reduces its structure group to $T^\CC$, which is abelian).

Let us define the closed subset of $(\hat{X} \otimes_\ZZ \Lambda_T) \times \mM^\CC$ given by the union of the irreducible subvarieties,
\begin{equation}
\label{eq definition of Xi_G0}
\Xi'_{G} := \bigcup_{B \in \Upsilon} \left( (\hat{X} \otimes_\ZZ \Lambda_T)|_B \times \aA_B^\CC \right).
\end{equation}
Recalling the action of $\Gamma_B$ on $\aA_B^\CC$ defined in \eqref{eq action of Gamma on aA_B}, set also
\begin{equation}
\label{eq definition of olXi_G0}
\Xi_{G} := \bigcup_{B \in \Upsilon} \left( (\hat{X} \otimes_\ZZ \Lambda_T)|_B \times \left ( \quotient{\aA_B^\CC}{\Gamma_B} \right ) \right).
\end{equation}

By construction of $\Xi'_G$, one has the obvious projection
\begin{equation} \label{eq Xi_G to X otimes Lambda}
\map{\Xi'_G}{\hat{X} \otimes_\ZZ \Lambda_T}{(t,z)}{t.}{}
\end{equation}
Let $\Upsilon_t$ be the set of admissible root systems $\{ B_1, \dots, B_\ell \}$, where $B_i \in \Upsilon_t$ if and only if
\[
t \in (\hat{X} \otimes_\ZZ \Lambda_T)|_{B_i}.
\]
For every $B_i \in \Upsilon_t$, there is only one admissible root system $D_j \in \Upsilon_t$ which contains $B_i$ as a subset and is not contained in any other element of $\Upsilon_t$. We say that such an element is {\it maximal} in $\Upsilon_t$, and we denote the set of all of them by
\[
\Upsilon_t^{max} = \{D_1, \dots, D_k \in \Upsilon_t \, : \, D_j \text{ is maximal in } \Upsilon_t  \}.
\]

\begin{remark} \label{rm fibre of Xi to X otimes Lambda}
Every $\aA_{B_i}^\CC$ with $B_i \in \Upsilon_t$ is contained in some $\aA_{D_j}^\CC$ for some $D_j \in \Upsilon_t^{max}$. Thus, the fibre of \eqref{eq Xi_G to X otimes Lambda} over $t$ is precisely the union $\bigcup_{D_j \in \Upsilon_t^{max}} \aA^\CC_{D_j}$. 
\end{remark}

We define the family of $G$-Higgs bundles parametrized by $\Xi'_G$
\begin{equation} \label{eq definition of Hh}
\Hh \lra X \times \Xi'_{G},
\end{equation}
setting for every $(t,z) \in \Xi'_{G}$,
\[
\Hh|_{(t,z)} := (\Ee_H|_{t}, z \otimes \s).
\]
Recall that $t \in (\hat{X} \otimes_\ZZ \Lambda_T)|_B$ and $z \in \aA_B^\CC$ for some $B \in \Upsilon$.

\begin{remark} \label{rm Xi is a surjective family of polystable}
By Remark \ref{rm Ee_H0 reduces to T_B}, $\Hh$ is a well defined family of polystable $G$-Higgs bundles. By Proposition \ref{pr E Phi reduces to a Cartan subgroup of G}, every polystable $G$-Higgs bundle of trivial characteristic class is isomorphic to $\Hh|_{(t,z)}$ for some $(t,z) \in \Xi'_G$.
\end{remark}

Recall from Remark \ref{rm W preserves I_nc} that $W$ preserves $I_{nc}(\gG,\cC_0)$ and therefore $\Upsilon$. Then, $\omega \in W$ sends $\aA_B^\CC$ to $\aA^\CC_{\omega \cdot B}$, and if $\omega$ lies in $Z_W(B)$ and therefore $\omega$ normalizes $\aA_B^\CC$, we set 
\[
\omega \cdot (x^\alpha + \ol{x}^\alpha) = x^{\omega \cdot \alpha} + \ol{x}^{\omega \cdot \alpha}.
\]
This allows us to extend the action of $W$ on $\hat{X} \otimes_\ZZ \Lambda_T$ to $\Xi'_{G}$ and further to $\Xi_G$. Two points of $\Xi_G$ related by the action of $W$ are conjugate by some element of $H^\CC$.

The family $\Hh \to X \times \Xi'_G$ of topologically trivial polystable $G$-Higgs bundles comes naturally with a framing at $x_0$. Take the map to the representation space induced by $\Hh$ and this framing,
\begin{equation} \label{eq nu}
\nu_G : \Xi'_G \to \Ff_X(G,x_0). 
\end{equation} 

In the next lemma we identify a closed subscheme of the moduli space containing the reduced subscheme. 

\begin{lemma} \label{lm description of RG^ps}
The image of $\nu_G$ is closed inside $\Ff_X(H,x_0)$. Furthermore, 
\begin{equation} \label{eq nu restricted to Xi is an isomorphism}
\nu_G(\Xi'_G) \cong \Xi'_G.
\end{equation}
Also, the polystable locus is $\Ff_X(G,x_0)^{ps} = H^\CC \cdot \nu_G(\Xi'_G)$ and
\begin{equation} \label{eq description of MG = RG/H}
\Mm' : = \git{H^\CC \cdot \nu_G(\Xi'_G)}{H^\CC}
\end{equation}
is a closed subscheme of $\Mm_X(G)$ such that
\[
\Mm_X^{\red}(G) \subset \Mm' \subset \Mm_X(G).
\]
\end{lemma} 

\begin{proof}
Recall from Proposition \ref{pr RG^ps decomposes} that $\Ff_X(G,x_0)^{ps}$ is isomorphic to a closed subset of the direct product $\Ff_X(H,x_0)^{ps} \times \left ( \mM^\CC \otimes H^0(X, \Oo_X) \right )$. After this and \eqref{eq nu restricted to X otimes Lambda is an isomorphism}, we see that the projection of $\nu_G(\Xi'_G)$ to $\Ff_X(H,x_0)^{ps}$ is isomorphic to $\hat{X} \otimes_\ZZ \Lambda_T$. It also follows from this decomposition that the fibre over $\nu_G((\hat{X} \otimes_\ZZ \Lambda_T)|_B)$ is precisely $\aA_B^\CC$. This proves \eqref{eq nu restricted to Xi is an isomorphism} and the closedness of $\nu_G(\Xi'_G)$. 

The bundles parametrized by $\Hh$ are polystable, so the image of $\nu_G$ is contained in $\Ff_X(G,x_0)^{ps}$ as well as $H^\CC \cdot \nu_G(\Xi'_G)$. In fact, since $H^\CC \cdot \nu_G(\Xi'_G)$ is closed and $H^\CC$-invariant, one has that 
\[
\Mm' \subset \Mm_X(G)
\]
is a closed subscheme. But by Remark \ref{rm Xi is a surjective family of polystable} every polystable $G$-Higgs bundle is contained in $H^\CC \cdot \nu_G(\Xi'_G)$. Then, $\Mm'$ contains every closed point of $\Mm_X(G)$ so it contains the reduced subscheme of $\Mm_X(G)$.
\end{proof}

We can now address the main theorem of the article.

\begin{theorem} \label{tm Mm cong olXi over W}
Let $G$ be a connected real form of the complex semisimple Lie group $G^\CC$ and let $X$ be an elliptic curve. The reduced moduli space of topologically trivial $G$-Higgs bundles over $X$ is 
\begin{equation} \label{eq Mm cong olXi over W}
\Mm_X^{\red}(G) \cong \quotient{\Xi_{G}}{W}.
\end{equation}
\end{theorem} 

\begin{proof}
Recalling that the action of $H^\CC$ over the polystable locus is free, we have from the description of \eqref{eq description of MG = RG/H}, that
\[
\Mm' = \git{H^\CC \cdot \nu_G(\Xi'_G)}{H^\CC} = \quotient{H^\CC \cdot \nu_G(\Xi'_G)}{H^\CC.} 
\] 

The action of the $\Gamma_B$ on \eqref{eq definition of olXi_G0} comes from the conjugation by elements in $T^\CC$. Also, since two points of $\Xi_G$ related by the action of $W$ are conjugate by some element of $H^\CC$, the morphism \eqref{eq nu} induces
\begin{equation} \label{eq olXi_G quotiented by W to Mm}
\quotient{\Xi_G}{W} \lra \quotient{H^\CC \cdot \nu_G(\Xi'_G)}{H^\CC} = \Mm'.
\end{equation}

By \eqref{eq nu restricted to X otimes Lambda is an isomorphism}, the construction of $\Xi'_G$ and Proposition \ref{pr RG^ps decomposes}, it follows that
\begin{equation} \label{eq nu Xi_G is an isomorphism}
\nu_G(\Xi'_G) \cong \Xi'_G.                                                       \end{equation}
Then, due to the universality of the quotients, the morphism \eqref{eq olXi_G quotiented by W to Mm} is an isomorphism provided it is bijective.

The next part of the proof is devoted to proving bijectivity of \eqref{eq olXi_G quotiented by W to Mm}. Thanks to the projection \eqref{eq Xi_G to X otimes Lambda}, one can construct
\[
\quotient{\Xi_G}{W} \lra \quotient{\hat{X} \otimes_\ZZ \Lambda_T}{W,}
\]
and the following diagram commutes
\[
\xymatrix{
\quotient{\Xi_{G}}{W} \ar[rr] \ar[d] & & \quotient{H^\CC \cdot \nu_G(\Xi'_G)}{H^\CC} \ar[d]
\\
\quotient{\hat{X} \otimes_\ZZ \Lambda_T}{W} \ar[rr] & & \quotient{H^\CC \cdot \nu_H(\hat{X} \otimes_\ZZ \Lambda_T)}{H^\CC.}
}
\]
Since the bottom row morphism is an isomorphism by \eqref{eq description of RH} and Theorem \ref{tm MmH}, it is enough to prove bijectivity on the fibres. 

Take a point $t \in \hat{X} \otimes_\ZZ \Lambda_T$ and take $\rho : \pi_1(X) \to T$ such that $\nu(t)$ corresponds to the polystable bundle $E_\rho$. By Remark \ref{rm fibre of Xi to X otimes Lambda}, the fibre of the left column morphism over $t$ is $\bigcup_{D_j \in \Upsilon_t^{max}} \aA^\CC_{D_j}$ quotiented by $Z_W(t)$, those elements that fix $t$. On the other column, by Proposition \ref{pr MmG onto MG}, the fibre associated to $E_\rho$ is $\zZ_{\mM^\CC}(\rho) \Slash Z_{H^\CC}(\rho)$. Then, the previous commu\-ting diagram restricts to
\begin{equation}
\label{eq bijection between the fibres}
\xymatrix{
\quotient{\bigcup_{D_j \in \Upsilon_t^{max}} \left( \aA_{D_j}^\CC/ \Gamma_{D_j}\right)}{Z_W(t)} \ar[rr] \ar[d] & & \raisebox{.2em}{\thinspace $\zZ_{\mM^\CC}(\rho)$} /\!\!/ \raisebox{-.2em}{$Z_{H^\CC}(\rho)$} \ar[d]
\\
[t]_{W} \ar[rr] & & [\nu(t)]_{H^\CC}.
}
\end{equation} 

Fixing $D_1 \in \Upsilon_t^{max}$, one has that
\begin{equation} \label{eq1}
\quotient{\bigcup_{D_j \in \Upsilon^{max}_t} (\aA_{D_j}^\CC/ \Gamma_{D_j})}{Z_W(t)} = \quotient{(\aA_{D_1}^\CC/ \Gamma_{D_1})}{Z_W(D_1) \cap Z_W(t).}
\end{equation}

By the choice of $\rho$, one has that $Z_W(t) = Z_W(\rho)$ and the maximal abelian subalgebra of $\zZ_{\mM^\CC}(\rho)$ is conjugate to  $\aA_{D_1}^\CC$ (and to every $\aA_{D_j}$ with $D_j \in \Upsilon^{max}_t$). 
Recall that $\aA_B^\CC$ is the maximal abelian subalgebra of $\mM_B^\CC$. The Real Chevalley Theorem (see for instance \cite[Theorem 6.57]{knapp}) allows us to express the GIT quotient in terms of a quotient of the maximal abelian subalgebra by the Small Weyl group 
\begin{equation} \label{eq2}
\raisebox{.2em}{\thinspace $\zZ_{\mM^\CC}(\rho)$} /\!\!/ \raisebox{-.2em}{$Z_{H^\CC}(\rho)$} \cong \quotient{\aA_{D_1}^\CC}{W_{sm} \left( Z_{H^\CC}(\rho), \aA_{D_1}^\CC \right).}
\end{equation}
Note that $\cC_{D_1}^\CC = \tT_{D_1}^\CC \oplus \aA_{D_1}^\CC$ is a Cartan subalgebra of $Z_{G^\CC}$. Let $Y^\rho_{D_1}$ be the Weyl group $W(Z_{G^\CC}(\rho), \cC_{D_1}^\CC)$. By \cite[Theorem 3]{konstant}, the small Weyl group is generated by the normalizer of $\aA_{D_1}^\CC$ in the Weyl group, that is 
\[
\quotient{\aA_{D_1}^\CC}{W_{sm} \left( Z_{H^\CC}(\rho), \aA_{D_1}^\CC \right)} \cong \quotient{\aA_{D_1}^\CC}{N_{Y^\rho_{D_1}}(\aA_{D_1}^\CC).} 
\]
Since $\aA_{D_1}^\CC$ is contained in $\zZ_{\mM^\CC}(\rho)$, we have that $\im \rho$ is contained in $T_{D_1} \subset T$ and therefore the Weyl group of $Z_{G^\CC}(\rho)$ is the subgroup of the Weyl group of $G^\CC$ that centralizes $\rho$,
\[
Y^\rho_{D_1} = Z_{Y_{D_1}}(\rho). 
\]
Then, 
\[
N_{Y^\rho_{D_1}}(\aA_{D_1}^\CC) = N_{Y_{D_1}}(\aA_{D_1}^\CC) \cap Z_{Y_{D_1}}(\rho).  
\]
By Lemma \ref{lm W_B cong Gamma_B rtimes Z_W B}, one has
\[
\quotient{\aA_{D_1}^\CC}{N_{Y_{D_1}}(\aA_{D_1}^\CC) \cap Z_{Y_{D_1}}(\rho)} \cong \quotient{\aA_{D_1}^\CC}{\left( \Gamma_{D_1} \rtimes Z_W(D_1) \right) \cap Z_W(\rho)}, 
\]
and therefore
\begin{equation} \label{eq3}
\quotient{\aA_{D_1}^\CC}{W_{sm} \left( Z_{H^\CC}(\rho), \aA_{D_1}^\CC \right)} \cong \quotient{(\aA_{D_1}^\CC/ \Gamma_{D_1})}{Z_W(D_1) \cap Z_W(\rho).} 
\end{equation}
Combining \eqref{eq1}, \eqref{eq2} and \eqref{eq3} one concludes that the top row morphism of \eqref{eq bijection between the fibres} is an isomorphism. Then \eqref{eq olXi_G quotiented by W to Mm} is an isomorphism,
\[
\Mm' \cong \quotient{\Xi_{G}}{W}.
\]
Hence, $\Mm'$ is reduced.

Finally, since $\Mm^{\red}_X(G) \subset \Mm' \subset \Mm_X(G)$ by Lemma \ref{lm description of RG^ps}, and $\Mm'$ is reduced, one has that $\Mm^{\red}(G) = \Mm'$ and the proof is complete.
\end{proof}

\begin{remark} \label{rm irreducible components of Mm}
Consider the family $\Hh_G \to X \times \Xi'_G$ and denote by $\Hh_B$ the restriction to $(\hat{X} \otimes_\ZZ \Lambda_T)|_B \times \aA_B^\CC$. Denote by $p_B$ the morphism to the moduli space induced by $\Hh_B$, given by moduli theory. Since $\Gamma_B \rtimes Z_W(B)$ normalizes $(\hat{X} \otimes_\ZZ \Lambda_T)|_B \times \aA_B^\CC$, after Theorem \ref{tm Mm cong olXi over W}, one has
\begin{equation} \label{eq description of the irreducible components of Mm}
p_B \left ( (\hat{X} \otimes_\ZZ \Lambda_T)|_B \times \aA_B^\CC \right ) \cong \quotient{\left ( (\hat{X} \otimes_\ZZ \Lambda_T)|_B \times \aA_B^\CC \right )}{\Gamma_B \rtimes Z_W(B)}.
\end{equation}
Each of the $p_B\left ( (\hat{X} \otimes_\ZZ \Lambda_T)|_B \times \aA_B^\CC \right )$ is irreducible, since $(\hat{X} \otimes_\ZZ \Lambda_T)|_B \times \aA_B^\CC$ is irreducible. We observe that $\Mm^{\red}_X(G)$ has the following decomposition into irreducible components,
\[
\Mm^{\red}_X(G) = \bigcup_{B \in \Upsilon} p_B \left ( (\hat{X} \otimes_\ZZ \Lambda_T)|_B \times \aA_B^\CC \right ).
\]
\end{remark}

\begin{remark} \label{rm description of MmH^CC}
In the case of complex semisimple Lie groups $H^\CC$ the Cartan decomposition is $\hH^\CC = \hH \oplus i \hH$. In that case, Cartan subalgebras have the form $\tT \oplus i\tT$, i.e. the non-compact part of a Cartan subalgebra is $\aA_0 = i \tT$. Also, we note that there are no imaginary non-compact roots, since for every root $\alpha$, one has that $(\hH^\CC)^\alpha$ is contained in the complexification of the compact subalgebra $\hH^\CC$ (which in this case coincides with the total Lie algebra). Then
\[
\Xi'_{H^\CC} = (\hat{X} \otimes_\ZZ \Lambda_T) \times \aA_0^\CC = (\hat{X} \otimes_\ZZ \Lambda_T) \times \tT^\CC \cong (T^*\hat{X} \otimes_\ZZ \Lambda_T),
\]
where we recall that $T^* \hat{X}  \cong \hat{X} \times \CC$ and $\tT^\CC \cong \CC \otimes_\ZZ \Lambda_T$, by the differential of \eqref{eq definition of Psi}. We also observe that, in this case, $\Xi_{H^\CC} = \Xi'_{H^\CC}$ by construction. Then, \eqref{eq Mm cong olXi over W} becomes
\[
\Mm^{\red}_X(H^\CC) \cong \quotient{(T^*\hat{X} \otimes_\ZZ \Lambda_T)}{W}.
\]
and we recover the description of \cite{thaddeus}. 
\end{remark} 

\begin{remark}
\label{rm MmUpq}
For the group $G = \SU(1,1)$ one has that $H = \nr{S}(\U(1) \times \U(1)) \cong \U(1)$ and therefore $W = \{ 1 \}$. The complexification of $\SU(1,1)$ is $\SL(2,\CC)$, which has a single (imaginary non-compact) root $\alpha$. Since $T = H = \U(1)$ we have that $\Lambda_T \cong \lambda \cdot \ZZ$, where $\lambda = 2 i \alpha^*$ is the generator of $\Lambda_T$. Then $\hat{X} \otimes_\ZZ \Lambda_T \cong X$. Also, one has that 
\[
\morph{\eta_\alpha}{\hat{X} \otimes_\ZZ \Lambda_T}{\hat{X}}{L \otimes \lambda}{L^{\otimes \alpha(\lambda) i^{-1}} = L^2,}{}
\]
so the preimage $\eta_\alpha^{-1}(\Oo_X)$ is the subset $\hat{X}[2] \subset \hat{X}$ of points of order $2$ (square roots of $\Oo_X$). To obtain $\Xi'_{\SU(1,1)}$ we glue a copy of $\CC$ over the points of this set. Therefore $\Xi_{\SU(1,1)}$ is $\hat{X} \cup (\hat{X}[2] \times \CC/ \pm)$ and, since $W$ is trivial, Theorem \ref{tm Mm cong olXi over W} gives the following isomorphism,
\[
\Mm^{\red}_X(\SU(1,1)) \cong \hat{X} \cup (\hat{X}[2] \times \CC/ \pm).
\]
This moduli space is not normal since the singular locus has codimension $1$.
\end{remark}

\begin{corollary} 
\label{co dimension of MmG_0}
The dimension of the moduli space of topologically trivial $G$-Higgs bundles is
\[
\dim_\CC(\Mm_X(G)) = \dim_\CC(\Mm^{\red}_X(G)) = \rk(G).
\]
\end{corollary} 

\begin{proof}
Take the dense open subset $U \subset (\hat{X} \otimes_\ZZ \Lambda_T)$ defined as the complement of the union of all $(\hat{X} \otimes_\ZZ \Lambda_T)|_B$ for every non-zero $B \in \Upsilon$. By construction, we have that $\Xi_G|_U = U \times \aA_0^\CC$, so
\begin{align*}
\dim_\CC(\Xi_G) & = \dim_\CC(\Xi_G|U)
\\
& = \dim_\CC(U) + \dim_\CC(\aA_0^\CC) 
\\
& = \dim_\CC(\hat{X} \otimes_\ZZ \Lambda_T) + \dim_\CC(\aA_0^\CC) 
\\
& = \dim_\CC(\tT^\CC) + \dim_\CC(\aA_0^\CC)
\\
& = \dim_\CC(\cC_0^\CC) 
\\
& = \rk(G).
\end{align*}
The group $W$ is finite, so taking the quotient by its action preserves the dimension.
\end{proof}

\section{Fixed points of involutions}
\label{sc involution}

Let $\sigma_G : G^\CC \to G^\CC$ be the involution defining the connected real
form $G$. As  stated in 
\cite{garcia,garcia-prada&ramanan}, this defines a holomorphic involution in 
the moduli space of $G^\CC$-Higgs bundles given by 
\[
\morph{\imath_G}{\Mm_X(G^\CC)}{\Mm_X(G^\CC)}{(E,\Phi)}{\left ( (\sigma_G)_*E, -(d\sigma_G)_*\Phi \right ).}{} 
\]


Taking the extension of structure group associated to $H^\CC \hookrightarrow G^\CC$ and the inclusion $\mM^\CC \subset \gG^\CC$ we construct an \'etale morphism
\begin{equation} \label{eq etale morphism j}
j: \Mm_X(G) \longrightarrow \Mm_X(G^\CC)^{\imath_G}.
\end{equation}

Take a maximally compact Cartan subalgebra $\cC_0 = \tT \oplus \aA_0$ of $\gG$ and let $C^\CC$ be the associated Cartan subgroup of $G^\CC$. Write $\Lambda_C$ for the cocharacter lattice of $G^\CC$, $\Lambda_{C_0} = \Hom(\CC^*, G^\CC)$, and write $Y$ for the Weyl group $W(G^\CC, C^\CC)$. Recall from Remark \ref{rm description of MmH^CC} that one has the isomorphism
\begin{equation} \label{eq description of MmG^CC}
\Mm^{\red}_X(G^\CC) \cong \quotient{T^*\hat{X} \otimes_\ZZ \Lambda_{C_0}}{Y.}
\end{equation}
In this section we study the involution $\imath_G$ in the context of this description.

\begin{remark}
Associated to the involution $\sigma_{\U(1)}$ on $\CC^*$, that sends $z$ to $\overline{z}^{-1}$, one has
\[
\morph{\imath_{\U(1)}}{T^*\hat{X}}{T^*\hat{X}}{(L,\phi)}{(L,-\phi)}{}
\]
\end{remark}

For any $\lambda \in \Lambda_{C_0}$, we define the holomorphic co\-cha\-rac\-ter $\sigma_G \cdot \lambda = \sigma_G \circ \lambda \circ \sigma_{\U(1)}$. By abuse of notation, we denote also by $\sigma_G$ the involution induced on $\Lambda_{C_0}$. This involution allows us to define
\[
\morph{\dot{\sigma}_G}{\CC^*\otimes_{\ZZ} \Lambda_{C_0}}{\CC^*\otimes_{\ZZ} \Lambda_{C_0}}{\sum z_i \otimes_{\ZZ} \lambda_i}{\sum \sigma_{\U(1)}(z_i) \otimes_{\ZZ} \sigma_G(\lambda_i).}{}
\]
We observe that $\dot{\sigma}_G$ corresponds with $\sigma_G$ via the isomorphism \eqref{eq definition of Psi}. Thus the diagram  
\begin{equation} \label{eq Psi commutes with sigma}
\xymatrix{
\CC^*\otimes_{\ZZ} \Lambda_{C_0} \ar[d]_{\dot{\sigma}_G}\ar[rr]^{\quad \cong} & & C_0^\CC \ar[d]^{\sigma_G}
\\
\CC^*\otimes_{\mathbb{Z}} \Lambda_{C_0}\ar[rr]^{\quad \cong} & & C_0^\CC}
\end{equation}
commutes. One can also check that the action of the Weyl group $Y = W(G^\CC, C_0^\CC)$ on $\Lambda_C$, extended to $\CC^* \otimes_\ZZ \Lambda_{C_0}$, commutes with the natural action of $\omega \in Y$ on $C_0^\CC$ under the isomorphism given in \eqref{eq definition of Psi},
\begin{equation} \label{eq Y commutes with Psi}
\xymatrix{
\CC^* \otimes_\ZZ \Lambda_{C_0} \ar[rr]^{\quad \cong} \ar[d]_{\omega \cdot} & & C_0^\CC \ar[d]^{\omega \cdot}
\\
\CC^* \otimes_\ZZ \Lambda_{C_0} \ar[rr]^{\quad \cong} & & C_0^\CC.
}
\end{equation}

In accordance with the definition of $\dot{\sigma}_G$, we set
\begin{equation} \label{eq definition of i_G}
\morph{i_G}{T^*\hat{X} \otimes_\ZZ \Lambda_{C_0}}{T^*\hat{X} \otimes_\ZZ \Lambda_{C_0}}{\sum (L_i, \phi_i) \otimes_{\ZZ} \lambda_i}{\sum \imath_{\U(1)}(L_i, \phi_i) \otimes_{\ZZ} \sigma_G \cdot \lambda_i}{}
\end{equation}
The commutativity of \eqref{eq Y commutes with Psi} implies that $\imath_G$ commutes with the action of $\omega \in Y$. Therefore this involution descends to the quotient by $Y$.

\begin{proposition} \label{pr description of i_G}
Under the isomorphism \eqref{eq description of MmG^CC}, $\imath_G$ is identified with
\begin{equation} 
\begin{array}{cccc} \imath_G \, : &  \Mm^{\red}_X(G^\CC)  & \longrightarrow &  \Mm^{\red}_X(G^\CC)  
\\ 
& \rotatebox{270}{$\cong$ \,} & & \rotatebox{270}{$\cong$ \,} 
\\ &  \quotient{T^*\hat{X} \otimes_\ZZ \Lambda_{C_0}}{Y}  & \longrightarrow &  \quotient{T^*\hat{X} \otimes_\ZZ \Lambda_{C_0}}{Y} \\ 
& \left [ (t,z) \right ]_Y &\longmapsto & \left [ i_G (t,z) \right ]_Y \end{array}
\end{equation}
Hence, for every $\omega \in Y$, the diagram
\begin{equation} \label{eq i commutes with xi}
\xymatrix{
T^*\hat{X} \otimes_\ZZ \Lambda_{C_0} \ar@{->>}[rr]^{p_0} \ar[d]_{i_G \circ \omega} & & \Mm^{\red}_X(G^\CC) \ar[d]^{\imath_G}
\\
T^*\hat{X} \otimes_\ZZ \Lambda_{C_0} \ar@{->>}[rr]^{p_0} & & \Mm^{\red}_X(G^\CC)
}
\end{equation}
commutes, $p_0$ being the projection induced by \eqref{eq description of MmG^CC}. 
\end{proposition}

\begin{proof}
The statements follow easily from the definition of $\dot{\sigma}_G$, the commutativity of \eqref{eq Psi commutes with sigma} and the construction of the isomorphism \eqref{eq description of MmG^CC} worked out in Remark \ref{rm description of MmH^CC}. 
\end{proof}

As a corollary, one has a description of the fixed point set $\Mm^{\red}_X(G^\CC)^{\imath_G}$.

\begin{corollary}
The fixed point set $\Mm^{\red}_X(G^\CC)^{\imath_G}$ is the union of all the projections of closed subsets $(T^*\hat{X} \otimes_\ZZ \Lambda_{C_0})^{i_G \circ \omega}$ given by the fixed points of the automorphisms $i_G \circ \omega$,
\[
\Mm^{\red}_X(G^\CC)^{\imath_G} = \bigcup_{\omega \in Y} p_{0} \left ( (T^*\hat{X} \otimes_\ZZ \Lambda_{C_0})^{i_G \circ \omega} \right ).
\]
\end{corollary}

The next step is to study the fixed points of the automorphisms $i_G$ and $i_G \circ \omega$ for every $\omega \in Y$. To do so, first we have to study the involution $\sigma_G : \Lambda_C \to \Lambda_C$ and the automorphisms $\sigma_G \circ \omega: \Lambda_{C_0} \to \Lambda_{C_0}$. Since
\[
(\sigma_G \circ \omega)^2 = \sigma_G \circ \omega \circ \sigma_G \circ \omega = \sigma_G(\omega) \circ \sigma_G^2 \circ \omega = \sigma_G(\omega) \omega,
\]
we observe that $\sigma_G \circ \omega$ is an involution if and only if 
\begin{equation} \label{eq involution condition}
\sigma_G(\omega) \omega = \id.
\end{equation}

Note that the projection $p_0$ preserves the dimension since it is given by a finite quotient. Then, the dimension of each of the components of the fixed locus is
\[
\dim_\CC \left ( p_0 (T^*\hat{X} \otimes_\ZZ \Lambda_{C_0})^{i_G \circ \omega} \right ) = \frac{\dim_\CC(T^*\hat{X} \otimes_\ZZ \Lambda_{C_0})}{\ord(\sigma_G \circ \omega)} = \frac{2 \rk(G^\CC)}{\ord(\sigma_G \circ \omega)}.
\]
When $\sigma_G \circ \omega$ is an involution, $\ord(\sigma_G \circ \omega) = 2$ and therefore the dimension of the fixed locus is $\rk(G^\CC)$. If \eqref{eq involution condition} is not satisfied, the order of the automorphism $\sigma_G \circ \omega$ is greater than $2$ and then, the dimension of the fixed locus is lower than $\rk(G^\CC)$.  We consider the union of all the components with maximal dimension, i.e. those components given by $\sigma_G \circ \omega$ satisfying \eqref{eq involution condition} 
\[
\Mm^{\red}_X(G^\CC)_{max}^{\imath_G} : = \bigcup_{\sigma_G(\omega) \omega = \id} p_{0} \left ( (T^*\hat{X} \otimes_\ZZ \Lambda_{C_0})^{i_G \circ \omega} \right ).
\]

\begin{remark} \label{rm every sigma_G omega is sigma_G in B}
When \eqref{eq involution condition} is satisfied, $\sigma_G \circ \omega$ is an involution of $C_0^\CC$, and by Lemma \ref{lm Cayley transform give all Cartan subalgebras}, the fixed point set $(C_0^\CC)^{\sigma_G \circ \omega}$ is $H^\CC$-conjugate to $C_B$, for some $B \in \Upsilon$, which is the fixed point set of $\sigma_G : C_B^\CC \to C_B^\CC$. 
\end{remark}

One can prove the converse as well.

\begin{lemma} \label{lm every sigma_G in B is sigma_G omega}
For every $B \in \Upsilon$ there exists $\omega \in Y$ satisfying \eqref{eq involution condition}, such that the involution $\sigma_G : C_B^\CC \to C_B^\CC$ is conjugate to the involution $\sigma_G \circ \omega : C_0^\CC \to C_0^\CC$.
\end{lemma}

\begin{proof}
Since $C_B^\CC$ and $C_0^\CC$ are conjugate by some element $g \in G^\CC$, one has that $\sigma_G : C_B^\CC \to C_B^\CC$ is conjugate to $\ad_g \circ \sigma_G \ad_{g}^{-1} : C_0^\CC \to C_0^\CC$. Note that
\[
\ad_g \circ \sigma_G \ad_{g}^{-1} = \sigma_G \circ \ad_{\sigma_G(g)} \circ \ad_{g}^{-1} = \sigma_G \circ \ad_{\sigma_G(g)g^{-1}}. 
\]
By hypothesis, $\sigma_G$ preserves $C_0^\CC$, so $\sigma_G(g)g^{-1}$ belongs to the normalizer $N_{G^\CC}(C_0^\CC)$ and therefore it defines an element of the Weyl group $Y = W(G^\CC, C_0^\CC)$. Furthermore,
\[
\sigma_G \left ( \sigma_G(g)g^{-1} \right ) \cdot \left (  \sigma_G(g)g^{-1} \right ) = \sigma_G^2(g) \sigma_G(g^{-1})\sigma_G(g) g^{-1} = g g^{-1} = \id,
\]
and \eqref{eq involution condition} is satisfied.
\end{proof}

Denote the cocharacter lattice of $C_B^\CC$ by $\Lambda_B := \Hom(\CC^*, \C_B^\CC)$ and construct the families of $G^\CC$-Higgs bundles $\Hh_{0} \to T^*\hat{X} \otimes_\ZZ \Lambda_{C_0}$ and $\Hh_{B} \to T^*\hat{X} \otimes_\ZZ \Lambda_{C_B}$ as we did in \eqref{eq definition of Hh} and Remark \ref{rm description of MmH^CC}. Denote by $q_0$, or $q_B$, the corresponding morphism from the parametrizing space of the family to the reduced moduli space $\Mm^{\red}_X(G^\CC)$.

\begin{lemma} \label{lm is enougth to study T*X otimes Lambda_B}
Take $\omega \in Y$ satisfying \eqref{eq involution condition}. Then we have the identification
\[
q_0 \left ( (T^*\hat{X} \otimes_\ZZ \Lambda_{C_0})^{i_G \circ \omega} \right ) = q_B \left ( (T^*\hat{X} \otimes_\ZZ \Lambda_{C_B})^{i_G} \right ).
\]
\end{lemma}

\begin{proof}
Since $i_G$ is induced by $\sigma_G$, the lemma follows from Remark \ref{rm every sigma_G omega is sigma_G in B}, Lemma \ref{lm every sigma_G in B is sigma_G omega}, and the fact that conjugation gives an isomorphism of $G^\CC$-Higgs bundles.
\end{proof}

Thus, to study $\Mm_X(G^\CC)^{\imath_G}_{max}$ we can reduce ourselves to the study of $i_G$ acting on $T^*\hat{X} \otimes_\ZZ \Lambda_{C_B}$.

\begin{corollary} \label{co decription of Mm_max^imath}
\[
\Mm^{\red}_X(G^\CC)_{max}^{\imath_G} = \bigcup_{B \in \Upsilon} q_B \left ( (T^*\hat{X} \otimes_\ZZ \Lambda_{C_B})^{i_G} \right ).
\]
\end{corollary}

Recall that $\cC_B^\CC \cong \CC \otimes_\ZZ \Lambda_{C_B}$ and $\tT_B^\CC \cong \CC \otimes_\ZZ \Lambda_{T_B}$, for the sublattice $\Lambda_{T_B} \subset \Lambda_{C_B}$. Consider also 
\[
\Lambda_{\aA_B} := \Hom(\CC^*, \exp(\aA_B^\CC)), 
\]
and note that $\aA_B^\CC \cong \CC \otimes_\ZZ \Lambda_{\aA_B}$. It follows that $\Lambda_{T_B} \oplus \Lambda_{\aA_B}$ has the same rank as $\Lambda_{C_B}$, although it might be a proper sublattice of $\Lambda_{C_B}$. 

The involution $\sigma_G$ leaves $\Lambda_{T_B}$ invariant, 
\begin{equation} \label{eq sigma_G to Lambda_T}
\sigma_G |_{\Lambda_{T_B}} = \id_{\Lambda_{T_B}},
\end{equation}
while $\sigma_G$ inverts $\Lambda_{\aA_B}$,
\begin{equation} \label{eq sigma_G to Lambda_aA}
\sigma_G |_{\Lambda_{\aA_B}} = - \id_{\Lambda_{\aA_B}}.
\end{equation}

We now study the \'etale morphism \eqref{eq etale morphism j}. Recall from Theorem \ref{tm Mm cong olXi over W} that $\Mm_X(G)$ is described in terms of $\Xi'_G$, defined in \eqref{eq definition of Xi_G0} as the union of the irreducible components $(\hat{X} \otimes_\ZZ \Lambda_T)|_B \times \aA_B^\CC$.

\begin{lemma} \label{lm irreducible component of T*X otimes Lambda^imath}
$(\hat{X} \otimes_\ZZ \Lambda_T)|_B \times \aA_B^\CC$ is an irreducible component of $(T^*\hat{X} \otimes_\ZZ \Lambda_{C_B})^{i_G}$.
\end{lemma}

\begin{proof}
Recalling that $\aA_B^\CC \cong \CC \otimes_\ZZ \Lambda_{\aA_B}$ and $(\hat{X} \otimes_\ZZ \Lambda_T)|_B \cong (\hat{X} \otimes_\ZZ \Lambda_{T_B})$, take the natural identification
\begin{equation} \label{eq fixed locus of i_0}
\left (\hat{X} \otimes_\ZZ \Lambda_T \right) |_B \times \aA_B^\CC \cong \left ( \left (\hat{X} \times \{ 0 \} \right ) \otimes_\ZZ \Lambda_{T_B} \right) \oplus \left ( \left ( \{ \Oo_X \} \times \CC \right )  \otimes_\ZZ \Lambda_{\aA_B} \right ).
\end{equation}
By \eqref{eq sigma_G to Lambda_T}, \eqref{eq sigma_G to Lambda_aA} and the definition of $i_G$ given in \eqref{eq definition of i_G}, one has that $i_G$ restricted to $T^*\hat{X} \otimes_\ZZ ( \Lambda_{T_B} \oplus \Lambda_{\aA_B})$ is
\[
\morph{i_G}{T^*\hat{X} \otimes_\ZZ ( \Lambda_{T_B} \oplus \Lambda_{\aA_B})}{T^*\hat{X} \otimes_\ZZ ( \Lambda_{T_B} \oplus \Lambda_{\aA_B})}{\sum (L_i, \phi_i) \otimes_\ZZ (\lambda_i \oplus \lambda_i')}{\sum (L_i, -\phi_i) \otimes_\ZZ (\lambda_i \oplus -\lambda_i').}{}{}
\]
Note that
\[
i_G \left ( \sum (L_i, \phi_i) \otimes_\ZZ (\lambda_i \oplus 0) \right ) = \sum (L_i, - \phi_i) \otimes_\ZZ (\lambda_i \oplus 0)
\]
and
\begin{align*}
i_G \left ( \sum (L_i, \phi_i) \otimes_\ZZ (0 \oplus \lambda'_i) \right ) = & \sum (L_i, - \phi_i) \otimes_\ZZ (0 \oplus - \lambda'_i) 
\\
= & \sum (L_i^*, \phi_i) \otimes_\ZZ (0 \oplus \lambda'_i).
\end{align*}
Thus, we see that \eqref{eq fixed locus of i_0} is fixed by $i_G$.
\end{proof}

Since $(\hat{X} \otimes_\ZZ \Lambda_T)|_B \times \aA_B^\CC$ is irreducible, $q_B(\hat{X} \otimes_\ZZ \Lambda_T)|_B \times \aA_B^\CC) \subset \Mm_X(G^\CC)$ is irreducible as well. By Corollary \ref{co decription of Mm_max^imath} and Lemma \ref{lm irreducible component of T*X otimes Lambda^imath}, the image of the \'etale map $j$ from \eqref{eq etale morphism j} is
\[
\xymatrix{
j \left ( \Mm^{\red}_X(G) \right ) = \bigcup_{B \in \Upsilon} q_B \left (\hat{X} \otimes_\ZZ \Lambda_T)|_B \times \aA_B^\CC \right ) \ar@{^{(}->}[d]
\\
\Mm^{\red}_X(G^\CC)_{max}^{\imath_G} = \bigcup_{B \in \Upsilon} q_B \left ( (T^*\hat{X} \otimes_\ZZ \Lambda_{C_B})^{i_G} \right ).
}
\]

Recall from \eqref{eq definition of Y_B} that $Y_B$ is the Weyl group associated to $C_B^\CC$ and note that the centralizer of $(\hat{X} \otimes_\ZZ \Lambda_T)|_B \times \aA_B^\CC$ in $Y_B$ coincides with $N_{Y_B}(\aA_B^\CC)$. Then, the projection of $(\hat{X} \otimes_\ZZ \Lambda_T)|_B \times \aA_B^\CC$ to $\Mm_X(G^\CC)^{\imath_G}_0$ is
\[
q_B \left ( (\hat{X} \otimes_\ZZ \Lambda_T)|_B \times \aA_B^\CC \right ) \cong \quotient{(\hat{X} \otimes_\ZZ \Lambda_T)|_B \times \aA_B^\CC}{N_{Y_B}(\aA_B^\CC)}.
\]
Recall from Remark \ref{rm irreducible components of Mm} that $\Mm_X(G)$ decomposes as the union of irreducible components $p_B \left ( (\hat{X} \otimes_\ZZ \Lambda_T)|_B \times \aA_B^\CC \right )$, where each of these components is described in \eqref{eq description of the irreducible components of Mm} as
\[
p_B \left ( (\hat{X} \otimes_\ZZ \Lambda_T)|_B \times \aA_B^\CC \right ) \cong \quotient{(\hat{X} \otimes_\ZZ \Lambda_T)|_B \times \aA_B^\CC}{\Gamma_B \rtimes Z_W(B).}
\]
Recall from Lemma \ref{lm W_B cong Gamma_B rtimes Z_W B} that $\Gamma_B \rtimes Z_W(B)$ can be identified as a subgroup of $N_{Y_B}(\aA_B^\CC)$.

Denote by $j_B$ the restriction of the \'etale morphism \eqref{eq etale morphism j} to the corresponding irreducible component. We study $j$ by describing each of the restrictions $j_B$.  

\begin{proposition} \label{pr description of embedding}
Denote by $\pi$ the natural projection induced by the identification of $\Gamma_B \rtimes Z_W(B)$ as a subgroup of $N_{Y_B}(\aA_B^\CC)$. One has the following commutative diagram
\[
\xymatrix{
\quotient{\left ( (\hat{X} \otimes_\ZZ \Lambda_T)|_B \times \aA_B^\CC \right )}{\Gamma_B \rtimes Z_W(B)} \ar@{->>}[r]^{\pi} \ar[d]^{\cong} & \quotient{\left ( (\hat{X} \otimes_\ZZ \Lambda_T)|_B \times \aA_B^\CC \right )}{N_{Y_B}(\aA_B^\CC)} \ar[d]^{\cong}
\\
p_B \left ( (\hat{X} \otimes_\ZZ \Lambda_T)|_B \times \aA_B^\CC \right ) \ar[r]^{j_B} \ar@{^{(}->}[d] & q_B \left ( (\hat{X} \otimes_\ZZ \Lambda_T)|_B \times \aA_B^\CC \right ) \ar@{^{(}->}[d]
\\
\Mm^{\red}_X(G) \ar[r]^j & \Mm^{\red}_X(G^\CC)^{\imath_G}.
}
\]
\end{proposition}

\begin{proof}
After the identifications that we have previously studied, the proof follows from the observation that $j$ is determined by the extension of structure groups given by the natural inclusions $H^\CC \subset G^\CC$ and $\mM^\CC \subset \gG^\CC$.
\end{proof}

\section{The Hitchin fibration}
\label{sc the hitchin fibration}

Let us consider the isotropy action of the complex reductive Lie group $H^\CC$ on $\mM^\CC$ and take the quotient map $\mM^\CC \to \git{\mM^\CC}{H^\CC}$. Let $E$ be any algebraic $H^\CC$-bundle. Since the isotropy action of $H^\CC$ on $\mM^\CC \Slash H^\CC$ is obviously trivial, we note that the fibre bundle induced by $E$ is trivial, $E(\mM^\CC \Slash H^\CC) = (\mM^\CC \Slash H^\CC) \otimes \Oo_X$, and so the projection induces a surjective morphism of fibre bundles
\[
E(\mM^\CC) \lra E(\mM^\CC \Slash H^\CC)
\]
and a morphism on the set of global sections
\[
\map{H^0(X,E(\mM^\CC))}{H^0(X,(\mM^\CC \Slash H^\CC) \otimes \Oo_X)}{\Phi}{\Phi \Slash H^\CC.}{}
\]
One can easily check that the map constructed above is constant along S-equivalence classes. This allows us to define the Hitchin map
\begin{equation}
\label{eq definition of b}
\morph{b_G}{\Mm_X(G)}{B_G := H^0(X,(\mM^\CC \Slash H^\CC) \otimes \Oo_X)}{(E,\Phi)}{\Phi \Slash H^\CC.}{}
\end{equation}

Let $D$ be a maximal admissible root system in $\Upsilon$, and let $\cC_0 = \tT \oplus \aA$ be the maximally non-compact Cartan subalgebra of $\gG$ associated to it. Then, $\aA^\CC$ is a maximal (abelian) subalgebra of $\mM^\CC$. Denote by $W_{sm}\left(H^\CC, \aA_{D}^\CC \right)$ the corres\-ponding small Weyl group. Recall again the Real Chevalley Theorem (\cite[Theorem 6.57]{knapp}, for instance), that states
\[
\raisebox{.2em}{\thinspace $\mM^\CC$} /\!\!/ \raisebox{-.2em}{$H^\CC$} \cong \quotient{\aA_{D}^\CC}{W_{sm} \left(H^\CC, \aA^\CC_{D} \right).}
\]
By \cite[Theorem 3]{konstant}, the small Weyl group is generated by the normalizer of $\aA_{D}^\CC$ in the Weyl group $Y$, and thanks to Lemma \ref{lm W_B cong Gamma_B rtimes Z_W B} one can identify $\aA^\CC / W_{sm}\left(H^\CC, \aA_{D}^\CC \right)$ with $\aA^\CC / \Gamma \rtimes Z_W(D)$. Since $X$ is a projective variety, we have 
\[
B_G = H^0(X,(\mM^\CC \Slash H^\CC) \otimes \Oo_X) \cong \quotient{\aA^\CC}{\Gamma \rtimes Z_W(D)}.  
\]

Recalling that $\Xi'_{G} \subset (\hat{X} \otimes_\ZZ \Lambda_T) \times \mM^\CC$, we consider the natural projection
\[
\morph{\pi_G}{\Xi'_{G}}{\bigcup_{B \in \Upsilon}\aA_B^\CC}{(t,z)}{z,}{}
\]
where $t \in \hat{X} \otimes_\ZZ \Lambda_T$ and $z \in \mM^\CC$. Due to Lemma \ref{lm Cayley transform give all Cartan subalgebras}, all the maximal admissible root systems $D_1, \dots, D_\ell$ are conjugate under $W$. Therefore, one has
\[
\quotient{\aA^\CC}{\Gamma \rtimes Z_W(D)} \cong \quotient{\left( \bigcup_{B \in \Upsilon}(\aA_B^\CC / \Gamma_B) \right)}{W,}
\]
where we recall that every $B$ is contained in some maximal $D$. One obtains the projection
\[
\beta_G : \left( \bigcup_{B \in \Upsilon}\aA_B^\CC \right) \, \lra \, \quotient{\aA^\CC}{\Gamma \rtimes Z_W(D)}.
\]

It is clear that the following diagram
\begin{equation}
\label{eq diagram Hitchin fibration}
\xymatrix{
\Xi'_{G}  \ar[rr]^{\pi_{G} \qquad} \ar[d]_{p_{G}} & & \bigcup_{B \in \Upsilon}\aA_B^\CC \ar[d]^{\beta_{G}}
\\
\Mm^{\red}_X(G)   \ar[rr]^{b_{G} \qquad} & & \quotient{\aA^\CC}{\Gamma \rtimes Z_W(D)}.
}
\end{equation}
is commutative.

Given $z \in \aA^\CC$, we define
\[
\Upsilon_z := \{ B \in \Upsilon \nr{ such that } z \in \aA_B^\CC  \}.
\]
We say that the admissible system $F$ is {\it minimal} in $\Upsilon_z$ if it does not contain any other admissible system of $\Upsilon_z$. Let $\Upsilon_z^{min}$ denote the set of minimal elements. Recall that, when $F \subset B$, one has $(\hat{X} \otimes \Lambda_T)|_B \subset (\hat{X} \otimes \Lambda_T)|_F$. Therefore, by \eqref{eq definition of X otimes Lambda restricted to B}, one has that
\begin{equation} \label{eq description of pi^-1}
\pi_G^{-1}(z) = \bigcup_{F_i \in \Upsilon_z^{min}} (\hat{X} \otimes \Lambda_T)|_{F_i}  \times \{ z \}.
\end{equation}

We can now describe explicitly the fibres of the Hitchin map restricted to $\Mm^{\red}_X(G)$.

\begin{lemma}
\label{lm form of the Hitchin fibre}
Let $z \in \aA^\CC$ and take $F \in \Upsilon_z^{min}$. Then, the Hitchin fibre in $\Mm^{\red}_X(G)$ over $z$ is
\[
b_{G}^{-1}(\beta_G(z)) \cong \quotient{(\hat{X} \otimes_\ZZ \Lambda_{T})|_{F}}{Z_{W}(F, z)}.
\]
\end{lemma} 

\begin{proof}
If we take any other minimal admissible root system $F' \in \Upsilon_z^{min}$, we observe that $|F| = |F'|$ so, by Lemma \ref{lm Cayley transform give all Cartan subalgebras}, $F'$ and $F$ are conjugate by the action of some element of $W$.

By \eqref{eq description of pi^-1}, since all the $F_{i}$ are related by the action of $W$, one has that the image under $p_{G}$ of $\pi_{G}^{-1}(z)$ is the quotient of one of the components, $(\hat{X} \otimes_\ZZ \Lambda_{T})|_{F}$, by the group that centralizes $z$ and preserves the component, which is $Z_W(z) \cap Z_W(F) = Z_{W}(F,z)$.
\end{proof}

\begin{corollary} 
The Hitchin fibre in $\Mm^{\red}_X(G)$ over $z \in \aA_0^\CC \subset \aA^\CC$ is
\[
b_{G}^{-1}(\beta_G(z)) \cong \quotient{(\hat{X} \otimes_\ZZ \Lambda_{T})}{Z_{W}(z)}.
\]
\end{corollary}

We can see that the set of $z \in \aA^\CC$ such that $z \notin \aA_B^\CC$ for any other $B \in \Upsilon$ such that $B \neq D$, is a dense open subset of $\aA^\CC$. If $z$ lies in this subset, we say that it is a {\it generic} element of $\aA^\CC$. Note that we have $\Upsilon_z = \{ D \}$ when $z$ is generic.

\begin{corollary}
\label{co generic Hitchin fibre}
The Hitchin fibre in $\Mm^{\red}_X(G)$ over a generic element $z \in \aA^\CC$ is
\[
b_{G}^{-1}(\beta_G(z)) \cong \quotient{(\hat{X} \otimes_\ZZ \Lambda_{T})|_D}{Z_{W}(D,z)}.
\]
\end{corollary}

\begin{remark}
The group $\SU^*(4)$ has $\Sp(4)$ as maximal compact subgroup and a unique (up to conjugation) Cartan subalgebra $\cC = \tT \oplus \aA$. By this uniqueness of the Cartan subalgebra, we know that $\Upsilon = \{ 0 \}$. If $\delta_i$ for $i = 1, \dots, 4$ are the elements with $1$ in the $i$-th position of the diagonal and $0$ elsewhere, we have that
\[
\tT = \RR \cdot i (\delta_2 - \delta_3) \oplus \RR \cdot \frac{i}{2}(\delta_1 - \delta_2 + \delta_3 - \delta_4).
\]
and
\[
\aA = \RR \cdot \frac{1}{2}(\delta_1 - \delta_2 - \delta_3 + \delta_4). 
\]
Note that any non-zero element of $\aA$ is a generic element. The kernel of $\exp_T$ is the lattice $\Lambda_T = \ZZ \cdot  i(\delta_2 - \delta_3) \oplus \ZZ \cdot i(\delta_1 - \delta_4)$.

The Weyl group $W= W(\Sp(4,\CC), \tT^\CC)$ is generated by the reflections of the roots $\delta^*_2 - \delta^*_3$, $\frac{1}{2}(\delta^*_1 - \delta^*_2 + \delta^*_3 - \delta^*_4)$ and $\delta^*_1 - \delta^*_4$. One can check that the centralizer of a generic element $z \neq 0$ of $\aA$ is $Z_W(z) = \langle \sigma_{14}, \sigma_{23} \rangle$ where $\sigma_{ij}$ is the permutation that sends $\delta_i$ to $\delta_j$ and leaves the rest unchanged. Applying Corollary \ref{co generic Hitchin fibre}, we have that the Hitchin fibre over $z$ is
\begin{align*}
b_{G}^{-1}(\beta_G(z)) \cong & \quotient{\hat{X} \otimes_\ZZ \Lambda_{T}}{Z_{W}(z)}
\\
= & \quotient{\hat{X} \otimes_\ZZ  \left( i(\delta_2 - \delta_3) \oplus \ZZ \cdot i(\delta_1 - \delta_4) \right)}{ \langle \sigma_{14}, \sigma_{23} \rangle}
\\
\cong & \left( \quotient{\hat{X} \otimes_\ZZ  \ZZ i(\delta_2 - \delta_3)}{\langle \sigma_{23} \rangle} \right) \times \left( \quotient{\hat{X} \otimes_\ZZ  \ZZ i(\delta_1 - \delta_4)}{\langle \sigma_{14} \rangle}\right)
\\
\cong & \left( \quotient{\hat{X} \otimes_\ZZ  \ZZ}{\pm} \right) \times \left( \quotient{\hat{X} \otimes_\ZZ \ZZ}{\pm}\right)
\\
\cong & \left( \quotient{X}{\pm} \right) \times \left( \quotient{X}{\pm} \right)
\\
\cong & \, \PP^1 \times \PP^1.
\end{align*}
\nr{It is remarkable that the generic fibre of the Hitchin fibration for $\SU^*(4)$ is $\PP^1 \times \PP^1$, which is not an abelian variety.}
\end{remark}

\section{The Hitchin equation and flat connections}
\label{sc hitchin equation}

Let $G$ be a connected real form of a complex semisimple Lie group $G^\CC$. Let $(E,\Phi)$ be a $G$-Higgs bundle and let $h$ be a metric on $E$, i.e. a $C^\infty$ reduction of $E$ to the maximal compact subgroup $H \subset H^\CC$ giving the $H$-bundle $E_h$. Hitchin introduced in \cite{hitchin-selfuality_equations} the so called {\it Hitchin equation} for a metric on a $G$-Higgs bundle.

Recall the Cartan involution $\theta$ and let $\theta_h : E_h(\gG^\CC) \to E_h(\gG^\CC)$ be the involution induced fibrewise by the Cartan involution. Let $\dolbeault_E$ denote the Dolbeault operator of $E$ and denote by $A_h$ the {\it Chern connection}, which is the unique $H$-connection on $E_h$ compatible with $\dolbeault_E$. We denote by $F_h$ the curvature of $A_h$. Take also $\dif x \in \Omega^{1,0}(X,\Oo_X)$ and $\dif \ol{x} \in \Omega^{0,1}(X,\Oo_X)$. In the case of an elliptic curve, the Hitchin equation reads
\begin{equation}
\label{eq Hitchin equation}
F_h + [\Phi \dif x, \theta_h (\Phi) \dif \ol{x}] = 0.
\end{equation}

We have seen in \cite[Proposition 5.1]{franco&garcia-prada&newstead_2} that the Hitchin equation splits in the case of elliptic curve and a complex reductive Lie group $G = H^\CC$. This result generalizes to semisimple real forms.

\begin{proposition} \label{pr splitting of the hitchin equation}
Let $G$ be a connected real form of the complex semisimple Lie group $G^\CC$ and let $H \subset G$ be a maximal compact subgroup. Fix a $G$-Higgs bundle $(E,\Phi)$. Suppose that either $H$ is semisimple or $E$ has trivial characteristic class.
Then $(E,\Phi)$ is polystable if and only if there exists a metric $h$ on
$E$ that satisfies 
\[
F_h =  0 \qquad \nr{and} \qquad [\Phi \dif x, \theta_h(\Phi) \dif
\ol{x}] = 0.
\]
\end{proposition}

\begin{proof}
By Proposition \ref{pr E Phi polystable implies E polystable}, if the $G$-Higgs bundle $(E,\Phi)$ is polystable, then $E$ is polystable and by the
Narasimhan--Seshadri--Ramanathan Theorem there exists a metric for which $F_h = 0$. 

By Corollary \ref{co description of polystable G-Higgs bundles}, $(E,\Phi)$ is
isomorphic to $(E_\rho,z \otimes \s)$ where $z \in \aA_\rho^\CC$ maximal abelian subalgebra of $\zZ_{\mM^\CC}(\rho)$. With no loss of generality, we can take $\aA_\rho^\CC$ to be contained in a $\theta$-stable Cartan subalgebra. Then, we have $[z,\theta(z)] = 0$ and
\begin{equation}  \label{eq vanishing of the Higgs field part}
[\Phi \dif x, \theta_h(\Phi) \dif \ol{x}] \cong [z,\theta(z)] \otimes \s \otimes (\dif x
\wedge \dif \ol{x}) = 0.
\end{equation}

Conversely, it follows from the fact that $F_h = 0$ defines a representation $\rho : \pi_1(X) \to H$ and then $(E, \Phi) \cong (E_\rho, z \otimes \s)$ where $z \in \zZ_{\mM^\CC}(\rho)$ by Corollary \ref{co description of semistable G-Higgs bundles}. Take $\aA^\CC_\rho$ to be a maximal abelian subalgebra of $\zZ_{\mM^\CC}(\rho)$ containing $z$ and $\theta(z)$. Finally, by Lemma \ref{lm complete description of polystable}, $(E_\rho, z \otimes \s)$ is polystable. 
\end{proof}

\begin{remark}
\label{rm Hitchin-Simpson easier}
Note that Proposition \ref{pr splitting of the hitchin equation} states the {\it Hitchin--Kobayashi correspondence}. In particular,  $\Mm_X(G)$  is homeomorphic to the moduli space $\Cc_X(G)$ of $G$-bundles with flat $G^\CC$-connections. Note also, that in order to prove the hard implication (``polystable implies existence of solutions of the Hitchin equation'') in the elliptic case, we make use only of the Narasimhan--Seshadri--Ramanathan Theorem and Proposition \ref{pr E Phi polystable implies E polystable}.
\end{remark}

\begin{proposition} \label{pr Toledo = 0}
Let $G$ be semisimple with maximal compact subgroup $H \subset G$ and let $(E, \Phi)$ be a semistable $G$-Higgs bundle. If $H$ has non-finite center, there are no solutions of the Hitchin equation \eqref{eq Hitchin equation} for $E$ with non-trivial characteristic class $d$ in $\pi_1(H^\CC)$. 
\end{proposition}

\begin{proof}
Note that, thanks to Corollary \ref{co description of semistable G-Higgs bundles}, the vanishing in \eqref{eq vanishing of the Higgs field part} still holds. Then, the Hitchin equation for a metric $h$ on $E$ forces
\[
F_h = 0,
\]
which has no solutions in the cases stated in the hypothesis.
\end{proof}

\begin{remark}
Let us take the semisimple real form $G = \SU(p,q)$ whose maximal compact subgroup $H = S(\U(p) \times \U(q))$ has non-finite center. By Proposition \ref{pr Toledo = 0} there are no solutions of the Hitchin equation unless $d \in \pi_1(H^\CC)$ is trivial. In this case $H^\CC = S(\GL(p,\CC) \times \GL(q,\CC))$, and then a $H^\CC$-bundle $E$ can be seen as a direct sum of two vector bundles $V \oplus W$ of rank $p$ and $q$ with $\det(V \oplus W) \cong \Oo_X$ and degrees $a = \deg(V)$ and $b = \deg(W)$ satisfying
\begin{equation} \label{eq a + b}
p a + q b = 0.
\end{equation}
The characteristic class $d \in \pi_1(H^\CC)$ is determined by $a$ and $b$ and Proposition \ref{pr Toledo = 0} implies that they are solutions of \eqref{eq Hitchin equation} only when $a = 0$ and $b = 0$. We can see that this agrees with the Milnor-Wood inequality for the Toledo invariant \cite{bradlow&garcia-prada&gothen}, 
\[
- 2 \min\{p, q\}(g-1) \leq \frac{pb - q a}{p q} \leq 2 \min\{p, q\}(g-1),
\]
which in this case reads
\begin{equation} \label{eq a - b}
pb - q a = 0.
\end{equation}
As we can see, $a = b = 0$ is the only possible solution of \eqref{eq a + b} and \eqref{eq a - b}.
\end{remark}

\end{document}